\documentclass{article}
\usepackage{amsmath,amsfonts,amssymb,amsthm,graphicx,epsfig,subfigure,float,url}
\usepackage[colorlinks=true]{hyperref}
\usepackage{pdfsync}
\usepackage{dsfont}
\usepackage[applemac]{inputenc}
\usepackage[usenames,dvipsnames]{pstricks}
\usepackage{multirow}
\usepackage{mathtools}
\usepackage{stmaryrd}

\newcommand{\R}{\mathrm{I\kern-0.21emR}}
\newcommand{\N}{\mathrm{I\kern-0.21emN}}

\newcommand{\nn}{\mathbf{n}}

\newcommand{\ff}{\mathbf{f}}
\newcommand{\Jac}{\mathbf{Jac}}
\newcommand{\eps}{\varepsilon}

\renewcommand{\red}[1]{\textcolor{black}{#1}}

\renewcommand{\geq}{\geqslant}
\renewcommand{\leq}{\leqslant}

\newtheorem{theorem}{Theorem}  
\newtheorem{proposition}{Proposition}
\newtheorem{corollary}{Corollary}
\newtheorem{definition}{Definition}
\newtheorem{lemma}{Lemma}

\theoremstyle{definition}\newtheorem{remark}{Remark}

\title{\bfseries Optimal releases for population replacement strategies, application to {\itshape Wolbachia}}

\author{L. Almeida\footnote{Sorbonne Universit\'e, CNRS,  UMR 7598, Laboratoire Jacques-Louis Lions, F-75005, Paris, France ({\tt luis.almeida@sorbonne-universite.fr}).}\and Y. Privat\footnote{IRMA, Universit\'e de Strasbourg, CNRS UMR 7501, \'Equipe TONUS, 7 rue Ren\'e Descartes, 67084 Strasbourg, France ({\tt yannick.privat@unistra.fr}).} \and M. Strugarek\footnote{AgroParisTech, 16 rue Claude Bernard, F-75231 Paris Cedex 05, France.}~\footnote{Sorbonne Universit\'{e}, Universit\'{e} Paris-Diderot SPC, CNRS, INRIA, Laboratoire Jacques-Louis Lions, \`Equipe Mamba, F-75005 Paris ({\tt martin.strugarek@ljll.math.upmc.fr}).} \and N. Vauchelet\footnote{LAGA - UMR 7539 Institut Galil\'{e}e Universit\'{e} Paris 13, 99 avenue Jean-Baptiste Cl\'{e}ment, 93430 Villetaneuse - France ({\tt vauchelet@math.univ-paris13.fr}).}
}

\date{}

\begin{document}

\maketitle

\begin{abstract}
In this article, we consider a simplified model of time dynamics for a mosquito population subject to the artificial introduction of {\itshape Wolbachia}-infected mosquitoes, in order to fight arboviruses transmission.
Indeed, it has been observed that when some mosquito populations are infected by some {\itshape Wolbachia} bacteria, various reproductive alterations are induced in mosquitoes, including cytoplasmic incompatibility. Some of these {\itshape Wolbachia} bacteria greatly reduce the ability of insects to become infected with viruses such as the dengue ones, cutting down their vector competence and thus effectively stopping local dengue transmission.

The behavior of infected and uninfected mosquitoes is assumed to be driven by a compartmental system enriched with the presence of an internal control source term standing for releases of infected mosquitoes, distributed in time. We model and design an optimal releasing control strategy with the help of a least square problem. In a nutshell, one wants to minimize the number of uninfected mosquitoes at a given horizon of time, under some relevant biological constraints.
We derive properties of optimal controls, highlight a limit problem providing useful asymptotic properties of optimal controls. We numerically illustrate the relevance of our approach.
\end{abstract}


\medskip


\tableofcontents

\section{Introduction}\label{sec:intro}

For many years (since \cite{W}), scientists have been studying {\itshape Wolbachia}, a bacterium living only inside insect cells. Recently, there has been increasing interest in the biology of {\itshape Wolbachia} and in its application as an agent for control of vector mosquito populations, by taking advantage of a phenomenon called {\itshape cytoplasmic incompatibility}. In key vector species such as {\itshape Aedes aegypti}, if a male mosquito infected with {\itshape Wolbachia} mates with a non-infected female, the embryos die early in development, in the first mitotic divisions (see \cite{Wer.Wolbachia}). This also happens even if the male and female are both infected with {\itshape Wolbachia} but are carrying mutually incompatible strains. Interestingly, an infected female can mate with an uninfected male producing healthy eggs just fine. Hence, using {\itshape cytoplasmic incompatibility} (CI) allows scientists to produce functionally sterile males that can be released in the field as an elimination tool against mosquitoes. This vector control method is known as incompatible insect technique (IIT).

Another promising application of this symbiotic bacteria is the control of endemic mosquito-borne diseases by means of population replacement. This control relies on the {\itshape pathogen interference} (PI) phenotype of some {\itshape Wolbachia} strains, especially with Zika, dengue and chikungunya viruses in {\itshape Aedes} mosquitoes (see \cite{Wal.wMel}). Population replacement methods have the benefit of having less immediate negative impact on the environment than insecticide-based approaches (since they are species specific) and potentially more cost effective (since they are long-lasting). Despite the broad range of arthropods carrying {\itshape Wolbachia}, no transmission event to any warm-blooded animals has been reported. The principle is to release {\itshape Wolbachia} carrying mosquitoes in endemic areas.
Once released, they breed with wild mosquitoes. Over time and if releases are large and long enough, one can expect the majority of mosquitoes to carry {\itshape Wolbachia}, thanks to CI. Due to PI, the mosquito population then has a reduced vector competence, decreasing the risk of Zika, dengue and chikungunya outbreaks. 

Both IIT and population replacement procedure have been imagined since a long time (see e.g. the work by Laven in 1967 \cite{laven} for population replacement, or the one by Curtis and Adak~\cite{curtis} in 1974 for population elimination, both on mosquitoes in genus {\itshape Culex}), but there has been a resurgence of interest lately for both techniques due to the increasing burden of arboviral diseases transmitted by mosquitoes in genus {\itshape Aedes}, and their operational implementation is a hot topic since the first report in \cite{Hof.Successful} of field success in Australian {\itshape Aedes aegypti} (see \cite{Lees2015} for IIT).
We focus here on population replacement strategies.

Motivated by the issue of controlling a population of wild {\itshape Aedes} mosquitoes by means of {\itshape Wolbachia} infected ones, we investigate here a simplified control model of population replacement strategies, where one acts on the wild population by means of time-distributed releases of infected individuals.
The evolution equations we use incorporate the competition of released individuals with the wild ones. Formally, let $n_1(t)$ denote the density of {\itshape Wolbachia}-free mosquitoes (the wild individuals) and $n_2(t)$ the density of {\itshape Wolbachia}-infected mosquitoes (the introduced ones) at time~$t$. 
We model population densitiy dynamics by the following competitive compartmental system:
\begin{equation}
\left\{
 \begin{array}{l}
\displaystyle  \frac{d n_1}{dt} (t) = f_1 (n_1(t), n_2(t)),
\\[10pt]
\displaystyle  \frac{d n_2}{dt} (t)= f_2 (n_1(t), n_2(t)) + u(t), \qquad t > 0, 
\\[10pt]
 n_1(0) = n_1^0, \quad n_2(0) = n_2^0, 
\end{array} \right.
\label{sys:general}
\end{equation}
where $u(\cdot)$ is a non-negative function standing for a {\bfseries control} (it models the release of {\itshape Wolbachia}-infected mosquitoes). The terms $f_i(n_1,n_2)$, $i=1,2$, are defined by
\begin{align}
\label{eq:n1}
f_1(n_1,n_2) =& b_1 n_1 \big( 1-s_h \frac{n_2}{n_1+n_2} \big) \big( 1-\frac{n_1+n_2}{K} \big) - d_1 n_1,   \\
\label{eq:n2}
f_2(n_1,n_2) =& b_2 n_2 \big( 1-\frac{n_1+n_2}{K} \big) -d_2 n_2.
\end{align}

The term $( 1-s_h \frac{n_2}{n_1+n_2})$ models the cytoplasmic incompatibility (CI): the parameter $s_h$ is the CI rate; one has $0\leq s_h \leq 1$ and when $s_h=1$, CI is perfect, whereas when $s_h=0$ there is no CI. The other parameters $(b_i, d_i)$ for $i \in \{1, 2\}$ are respectively intrinsic mortality and intrinsic birth rates, and $K$ denotes  the environmental carrying capacity.
A model such as \eqref{eq:n1}-\eqref{eq:n2} for mosquito population dynamics with {\itshape Wolbachia} has been introduced in \cite{Farkas2010,Fen.Solving}, and also studied \cite{HugBri.Modelling} where it was coupled with an epidemiological model. In \cite{ChaKim.Modeling}, similar dynamics have been described (including also a spatial dimension); further discussion on these various models can be found in \cite{SV2016}. We note that the addition of a control term was already proposed in \cite{Colombien} for population replacement and in \cite{SIT} for IIT (coupled with insecticide), where some associated optimization problems were described.

To make it closed, this system is complemented with nonnegative initial data $(n_1^0,n_2^0)$. We will assume to be, at time $t=0$, in the ``worst'' initial situation where there are no {\itshape Wolbachia}-infected mosquitoes in the population, in other words $n_2^0=0$.
When useful, we will use the notations
\[
 \nn = (n_1, n_2)\qquad \text{and}\qquad \ff= (f_1, f_2)
\]
to denote respectively the density mosquitoes vector and the right-hand side functions vector in~\eqref{eq:n1}-\eqref{eq:n2}.

The mathematical model \eqref{sys:general}-\eqref{eq:n1}-\eqref{eq:n2} in the absence of control (in other words when $u=0$) will be analyzed and commented in Section \ref{sec:systgalAnalysis}.
The starting point of our analysis is to notice that this system has, as steady states (in addition to the trivial one $(0,0)$)
\[
(n_1^*,0)\quad \text{and}\quad (0,n_2^*), \quad \text{with }n_i^*=K\left(1-\frac{d_i}{b_i}\right), \ i=1,2,
\]
corresponding to the invasion of the total population of mosquitoes, either by the wild one or the {\itshape Wolbachia}-infected one. In the following, we will make several assumptions guaranteeing that System \eqref{sys:general} is bistable and monotone. Our main objective is to build a strategy allowing us to reach the stable state $(0, n_2^*)$, starting from the other stable state $(n_1^*, 0)$, by determining {\itshape in an optimal way} a {\bfseries control law} $u (t)$. 
Any path leading from $(n_1^*, 0)$ to the basin of attraction of $(0, n_2^*)$ will be called a {\itshape population replacement strategy}.
Our aim is thus to steer the control system as closely as possible to the steady state $(0, n_2^*)$ at time $T > 0$. In an informal way, we investigate the following issue:
\begin{center}
{\sl How to design optimally the releases of {\itshape Wolbachia}-carrying mosquitoes (in other words, how to choose a good control function $u(\cdot)$) in order to favor the establishment of {\itshape Wolbachia} infection?}
\end{center}

Of course, to make this issue relevant, it is necessary to assume some constraints on the control function $u(\cdot)$, modeling in particular the fact that the ability of scientists to create {\itshape Wolbachia}-infected mosquitoes is limited. In the converse case, it is likely that a trivial answer would be to release the maximal possible number of mosquitoes at each time $t$. In the sequel, we will hence consider the following constraints (of pointwise and integral types) on the control function $u(\cdot)$
 \[
 0\leq  u(t) \leq M\text{ a.e. on }(0,T)\qquad\text{ and }\qquad\int_0^T u(t) \,  dt \leq C
 \]
 for some positive constants $M$ and $C$, meaning that the flux of {\itshape Wolbachia}-infected mosquitoes that can be released at each time $t$ is limited, as well as their total amount over the horizon of time $T$.
 
 In the analysis to follow, we use the essential property that System \eqref{sys:general}  is competitive, meaning that it enjoys a comparison principle (see Lemma \ref{lem.systMonot}).

From the mathematical point of view, problems considered in this article are related to optimal control theory for biological systems. Such kind of application has not been muchstudied so far. Nevertheless, we mention \cite{carrere,LLNP,trelatzhuzuazua} on optimal control problems for mono/bi-stable systems,  noting that this list is far from being exhaustive.

\bigskip

Let us describe our main results. 
When $b_1$, $b_2$ are large, we show that the proportion of {\itshape Wolbachia}-infected mosquitoes $p=n_2 / (n_1 + n_2)$ converges to the solution of a reduced problem of the form 
\begin{equation}
 \frac{dp}{dt}  = f(p) + u g(p),
 \label{eq:prop}
\end{equation}
with $g \geq 0$ and $f$ of bistable type\footnote{The wording ``{\bfseries bistable function}'' means that $f(0) = f(1) = 0$ and there exists $\theta \in (0, 1)$ such that $f(x)(x - \theta) < 0$ on $(0, 1) \setminus \{ \theta \}$ (in particular, one has necessarily $f(\theta)=0$ whenever $f$ is smooth).\label{footnote.bistable}}. Bistable frequency-based models such as \eqref{eq:prop} have been studied extensively (see in particular \cite{BarTur.Spatial,SCHRAIBER201226}) for cytoplasmic incompatibility modeling since the works of Caspari and Watson \cite{caspariwatson}. Yet, as a new feature \eqref{eq:prop} incorporates rigorously a control term. The typical control for this biological system being the releases of individuals, it was not straightforward to understand how that control would act on the proportion $p$ of infected individuals. Our approach thus provides a way to derive a relevant control system on $p$ from the standard control system~\eqref{sys:general} where the input is a density of released individuals.
We first prove that the optimization problems converge along with the equations ($\Gamma$-convergence result stated in Proposition \ref{prop:convergence}) to a limit problem, and then solve it completely (Theorem \ref{thm:reducedprob}). It appears that the solutions to the limit problem consist of a single release phase where the maximal flux capacity $M$ is used. Generically, this phase occurs either at the very beginning or at the very end of the time frame $[0, T]$, depending on whether the constraints allow for the existence of a population replacement strategy or not.

Numerical investigations illustrate this behavior and also hint that the optimal strategies for steering system \eqref{sys:general} toward infection establishment may differ significantly from those suitable for~\eqref{eq:prop}.


\bigskip

The article is organized as follows. 
Section \ref{sec.towardOCP} is devoted to modeling issues: we introduce the simplified dynamics we consider for the system of wild versus {\itshape Wolbachia}-carrying mosquitoes, as well as the optimal control problem \eqref{prob:full} used to design a release strategy.

This problem is then analyzed in Section \ref{sec.analysisPfull}. More precisely, we show in Section~\ref{sec.descript} that~\eqref{prob:full} and its solution converge to a population replacement strategy optimization problem~\eqref{prob:reduced} for the simplified model \eqref{eq:prop}, in the limit when birth rates are assumed to be large. Numerical experiments validating our approach are presented in Section \ref{sec:num}.
%


\section{Toward an optimal control problem}\label{sec.towardOCP}
\subsection{On the dynamics without control}\label{sec:systgalAnalysis}
First we describe precisely the asymptotic behavior of  System \eqref{sys:general} in the absence of control (in other words, when $u(\cdot)=0$).  An example of phase portrait illustrating this lemma is provided on Fig \ref{PPsyst}.
There and for all numerical illustrations of our results, the parameter values we choose for $b_i, d_i$ and $s_h$ reflect the effects of a {\itshape Wolbachia} infection in {\itshape Aedes} mosquitoes. In the well-documented case of the {\itshape Wolbachia} strain {\itshape wMel} in {\itshape Aedes aegypti} and according to \cite{Wal.wMel,Dut.Lab}, it is relevant to choose: slight fecundity reduction ($b_2 / b_1 \simeq 0.9$), slight life-span reduction ($d_2/d_1 \simeq 1.1$) and almost perfect CI ($s_h = 0.9$). We do not fix a time scale, hence the last biologically meaningful parameter is $b_1/d_1$, the basic reproduction number for the wild population. Freely inspiried by literature estimates (see \cite{Focks,Otero,HugBri.Modelling}) we assume that this number is large, at least equal to $3$ (and taking values in all the range $[3.7, 7400]$ in Section \ref{sec:num}). Since these values are used only for results illustration, they are not intended to represent precisely a well-identified mosquito population-{\itshape Wolbachia} strain couple.
\begin{lemma}\label{lem.systMonot}
System \eqref{sys:general} is positive and (monotone) competitive\footnote{This means that if $(n_1^{\pm},n_2^{\pm})$ are solutions of \eqref{sys:general} such that $n_1^- (0) < n_1^+ (0)$ and $n_2^- (0) > n_2^+ (0)$ then one has $n_1^-(t)<n_1^+(t)$ and $n_2^- (t) > n_2^+ (t)$ for every time $t\in [0,T]$, where $(n_1^-,n_2^-)$ (resp. $(n_1^+,n_2^+)$) denotes the solution of System \eqref{sys:general} associated to the choice of initial conditions $(n_1^0,n_2^0)=(n_1^-(0),n_2^-(0))$ (resp. $(n_1^0,n_2^0)=(n_1^+(0),n_2^+(0))$)\label{footnote.comp.pple}}.

Let us assume that
\begin{equation}
b_1 > d_1 \quad \text{ and } \quad b_2 > d_2 .
\label{cond:1}
\end{equation}
Then, System \eqref{eq:n1}-\eqref{eq:n2} with $u(\cdot) = 0$ has at least three non-negative steady states: 
\[
(0, 0), \quad (n_1^*, 0),\quad (0, n_2^*),\quad \text{with}\quad n_i^* = K \left(1 - \frac{d_i}{b_i}\right), \ i \in \{1, 2\}.
\]
In this case, each population can sustain itself in the absence of the other one. In addition, $(0, 0)$ is (locally linearly) unstable.

Moreover, there exists a fourth distinct positive steady state if and only if
\begin{equation}
 1 - s_h < \frac{d_1 b_2}{d_2 b_1} < 1.
 \label{cond:2}
\end{equation}
In this case, this coexistence equilibrium is (locally linearly) unstable, and is given by
\[
 \nn^C = K \left( \left( 1 - \frac{1}{s_h} \left(1 - \frac{d_1 b_2}{d_2 b_1}\right) \right) \left( 1 - \frac{d_2}{b_2} \right), \frac{1}{s_h} \left(1 - \frac{d_1 b_2}{d_2 b_1}\right) \left( 1 - \frac{d_2}{b_2} \right) \right).
\]
Moreover, the two other nontrivial steady states are locally asymptotically stable in this case.
\end{lemma}

For the sake of readability, the proof of this result is postponed to Appendix \ref{sec.proofs}. Notice that conditions \eqref{cond:1} and \eqref{cond:2} on the parameters are relevant since \textit{Wolbachia}-infected {\itshape Aedes} mosquitoes typically have (even slightly) reduced fecundity and lifespan (for instance in the case of {\itshape wMel} strain, \cite{Wal.wMel}). Moreover CI is almost perfect in these species-strain combination (see \cite{Dut.Lab}), \textit{i.e.} $s_h$ is close to $1$.

\begin{figure}[!ht]
\begin{center}
\includegraphics[width=12cm]{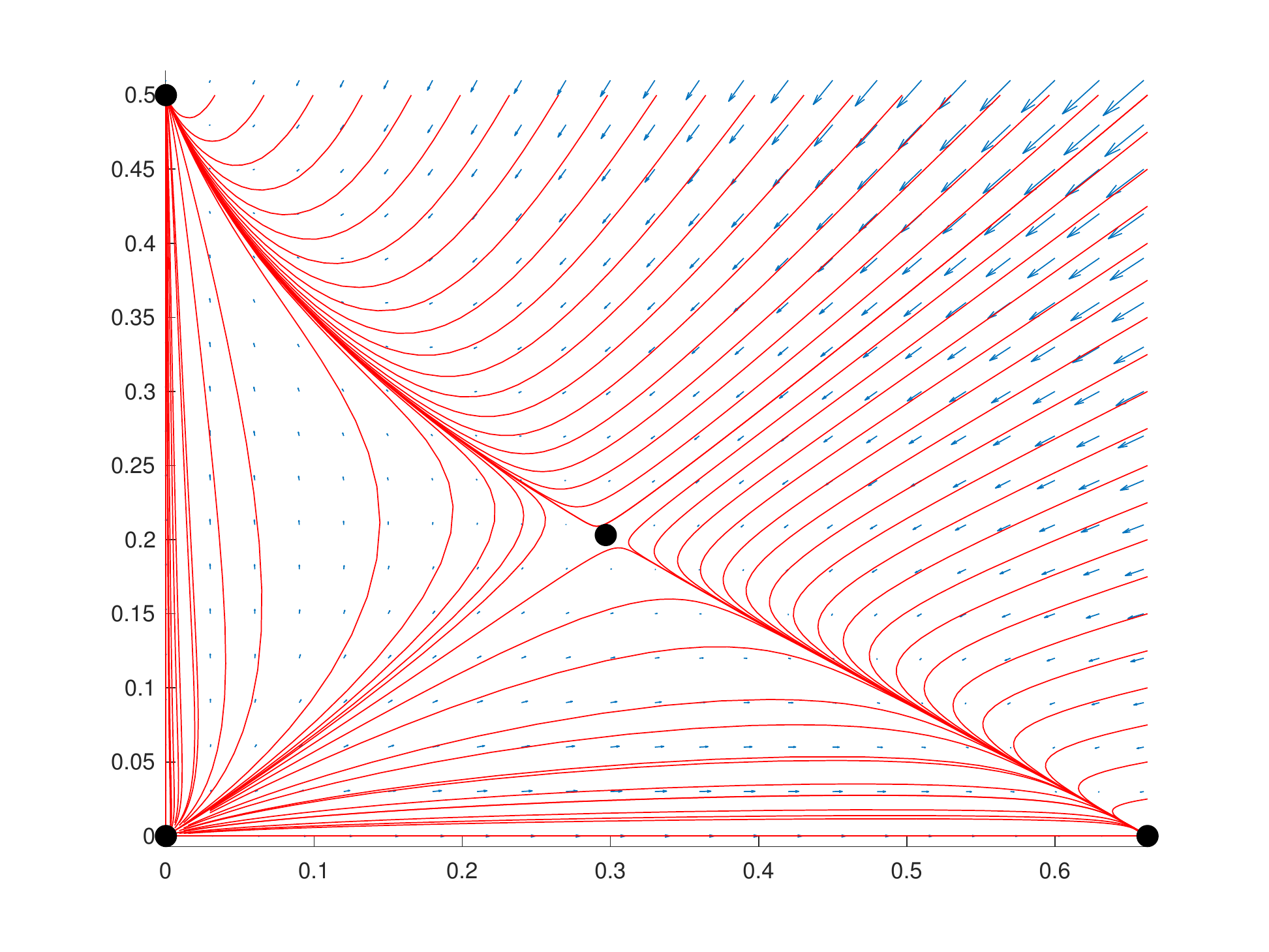}
\caption{Phase portrait of System \eqref{sys:general} for the parameters choice: $b_1=0.8$, $b_2=0.6$, $d_1=0.27$, $d_2=0.3$,  $s_h=0.8$ and $K=1$ for which conditions \eqref{cond:1} and \eqref{cond:2} are satisfied. Examples of trajectories are plotted with continuous lines. The dots locate the four steady states.\label{PPsyst}}
\end{center}
\end{figure}

\paragraph{Interpretation.} In short, under the biologically relevant conditions \eqref{cond:1} and \eqref{cond:2}, the two mutual exclusion steady states are stable while whole population extinction and coexistence state are unstable: in our model, either one of the two phenotypes must prevail in the long run, eliminating the other one.


\subsection{Objective function and constraints on the control}
Let us fix a horizon of time $T>0$. In this section, we propose a relevant choice of objective function $u\mapsto J(u)$, trying to model that we expect the control be chosen so that the final state (at time $T$) of System \eqref{sys:general} be as close as possible to the steady state $(0, n_2^*)$ corresponding to a population replacement situation. Since there is no obvious choice, we will consider a least square type functional, having the property to decrease as $\nn(T)$ gets closer to $(0, n_2^*)$.

This leads to introduce 
\begin{equation}\label{Ju}
J(u) = \frac 12 n_1(T)^2 + \frac 12 \Big[(n_2^*-n_2(T))_+\Big]^2,
\end{equation}
where, for $X\in \R$, the notation $X_+$ stands for $\max \{X,0\}$, and $\nn = (n_1, n_2)$ denotes the solution of \eqref{sys:general} associated to, in some sense, the worst initial data $\nn(0) = (n_1^*, 0)$.
Notice that, to ensure consistency of our model, any larger value of the introduced population than the equilibrium value $n_2^*$ is not detrimental for $J(u)$.
This objective function differs from the ones introduced in \cite{Colombien,Campo-Duarte2018}, where a $L^2$ norm is used to optimize a similar protocol of {\itshape Wolbachia} infection establishment by releases. Here, we are only interested in the state at the end of the treatment, which determines protocol success or failure.

\medskip

Let us enumerate the mathematical constraints we will assume on the control function $u(\cdot)$, stemming from biology.
\begin{itemize}
\item $u(t)$ corresponds to the density of {\itshape Wolbachia}-infected released mosquitoes and must be nonnegative (since we assume that we only release individuals and cannot remove them). 
\item Since System \eqref{sys:general} is monotone, it is relevant to assume an upper bound on the total number of released individuals, namely
\[
 \int_0^T u(t) dt \leq C
\]
for some given $C>0$. Indeed, releasing more and more individuals can never be detrimental. Without such a constraint, the solution of the considered optimal control problem is trivial and consists in releasing as much individuals as possible at each time.
\item For practical reasons, it is neither possible to create an infinite number of {\itshape Wolbachia}-infected individuals nor to release them ``instantly'' at time $t$. Hence, this leads to assume a pointwise upper bound on the control, by setting $u(t) \leq M$ for some $M>0$ and all $t \in [0, T]$. This constraint models that a release is necessarily distributed in time (possibly on a very short period of time) and cannot be an impulse.
\end{itemize}
All these considerations lead us to introduce the following set of admissible controls
\begin{equation}\label{def.UTCM}
\boxed{\mathcal{U}_{T,C,M} = \{u\in L^{\infty}([0,T]), \quad 0\leq u\leq M \text{ a.e. }, \int_0^T u(t)\,dt \leq C\}}.
\end{equation}

We then deal with the following optimal control problem.
\begin{equation}\label{prob:full}
\boxed{
\inf_{u\in \mathcal{U}_{T,C,M}} J(u).
}
\tag{$\mathcal{P}_{\text{full}}$}
\end{equation}
where $J$ is defined by \eqref{Ju} and $\mathcal{U}_{T,C,M}$ is defined by \eqref{def.UTCM}. 

\paragraph{Interpretation.} Problem \eqref{prob:full} amounts to finding a constrained release protocol (in terms of total number of released individuals and maximal release flux) which steers the system as close as possible to the target state: elimination of the wild phenotype and establishment of the introduced one.

\subsection{System and problem reductions}
\label{sec:reduction}
From a practical point of view, it appears relevant to consider that intrinsic birth rates are large compared with intrinsic death rates, since vector {\itshape Aedes} species typically have a very high reproductive power. For this reason, we will introduce (at the end of this section) and then analyze (in Section \ref{sec.analysisPfull}) a simplified version of Problem \eqref{prob:full} that will help to infer some interesting qualitative properties of the solution of Problem \eqref{prob:full}. This way, we will  reduce System \eqref{sys:general} into a simple scalar equation on the proportion of {\itshape Wolbachia}-infected mosquitoes in the spirit of \cite{SV2016}. To do so, let us introduce a small parameter $\eps > 0$ and the birth rates 
\begin{equation}\label{asymp.b1b2}
b_1 = b_1^0 / \eps\quad \text{ and }\quad b_2 = b_2^0 / \eps
\end{equation} 
for some positive numbers $b_1^0$, $b_2^0$.

It is notable that, in that case, the steady-states $(n_1^*,0)$ and $(0,n_2^*)$ respectively converge to $(K,0)$ and $(0,K)$ as $\eps\searrow 0$, since $n_i^* = K(1-\eps\frac{d_i}{b_i^0})$, $i=1,2$.
Notice also that \eqref{cond:1} is automatically satisfied as soon as $\eps$ is small enough.

In what follows, we will denote by $J^\eps$ the functional defined by 
\begin{equation}\label{Jueps}
J^\eps(u) = \frac 12 n_1^\eps(T)^2 + \frac 12 (n_2^*-n_2^\eps(T))_+^2,
\end{equation}
where $(n_1^{\eps}, n_2^{\eps})$ denote the solution to Problem \eqref{sys:general} with $b_1$ and $b_2$ given by \eqref{asymp.b1b2}. Let us introduce the variables 
\begin{equation}\label{def.Nepspeps}
N^{\eps} = n_1^{\eps} + n_2^{\eps}\qquad\text{ and }\qquad p^{\eps} = n_2^{\eps} / N^{\eps}.
\end{equation}
Setting $n^{\eps} = \frac{1}{\eps} \left(1 - \frac{N^{\eps}}{K}\right)$, we have the following (technical but crucial) convergence result, saying that the pair $(n^{\eps}, p^{\eps})$ converges in some sense to a well-identified limit $(u,p)$.
\begin{proposition}
Let $u^{\eps} \in \mathcal{U}_{T,C,M}$ such that $(u^{\eps})_{\eps>0}$ converges weakly-star
 to $u \in \mathcal{U}_{T,C,M}$  in $L^\infty(0,T)$ as $\eps\searrow 0$.
 
The pair $(n^{\eps}, p^{\eps})$ associated to the control $u^{\eps}$ and the parameter scaling \eqref{asymp.b1b2} solves a slow-fast system of the form
\begin{equation}\label{syst.nepspeps}
 \left\{
 \begin{array}{l}
  \eps \displaystyle \frac{d n^{\eps}}{dt} = (1 - \eps n^{\eps}) a(p^{\eps}) \big(Z (p^{\eps}) - n^{\eps} \big) - \frac{u^{\eps}}{K}, \\[10pt]
  \displaystyle \frac{d p^{\eps}}{dt} = p^{\eps} (1 - p^{\eps}) \big( n^{\eps} (b_2^0 - b_1^0 (1 - s_h p^{\eps})) + d_1 - d_2 \big)+ \frac{u^{\eps }(1-p^{\eps})}{K (1 - \eps n^{\eps})}, \quad t>0\\[10pt]
   \displaystyle n^{\eps} (0) = \frac{d_1^0 }{ b_1^0}, \  p^{\eps} (0) = 0,
 \end{array}
\right.
\end{equation}
where $a(p)$ and $Z(p)$ are defined by
\[
 a (p) = b_1^0 (1 - p) (1 - s_h p) + b_2^0 p > 0,
 \quad
 Z (p) = \frac{d_1 (1 - p) + d_2 p}{a(p)} > 0.
\]
Let us assume that \eqref{cond:2} holds and let $\eps_0 > 0$ be such that 
\begin{equation}
\frac{d_1}{ b_1^0} < \frac{1}{ \eps_0}\quad\text{ and }\quad \max_{[0, 1]} Z < \frac{1}{ \eps_0}.
\end{equation}
Then, for all $\eps \in (0, \eps_0)$, we have the uniform estimates 
\begin{equation}
0 \leq p^{\eps}(t) \leq 1\quad\text{ and }\quad n_-\leq n^{\eps}(t)\leq n_+
\end{equation}
for all $t\in [0,T]$ where
 \begin{eqnarray*}
n_-&=& \min \left\{\frac{d_1}{b_1^0}, \min_{\eps \in [0, \eps_0]} \min_{p \in [0, 1]} \frac{1+\eps Z(p) - \sqrt{(1-\eps Z(p))^2 + 4 \eps M / (K a(p))}}{2 \eps} \right\} \\
 n_+&=& \max \left\{ \frac{d_1}{b_1^0}, \max_{p \in [0, 1]} Z(p) \right\}.
\end{eqnarray*}
Moreover, up to a subfamily, $(p^{\eps})_{\eps>0}$ converges uniformly as $\eps\searrow 0$ to $ p$, solving
\begin{equation}\label{eq:p}
 \left\{
 \begin{array}{l}
\displaystyle \frac{dp}{dt} = f(p)+ug(p), \quad t>0 \\[10pt]
 p(0) = 0
 \end{array}
\right.
\end{equation}
where 
\begin{equation}\label{deffdefg}
f(p)=p(1-p)\frac{d_1b_2^0 - d_2 b_1^0 (1-s_h p)}{b_1^0 (1-p) (1-s_h p)+b_2^0p}\quad \text{and}\quad g(p)=\frac{1}{K}\cdot\frac{b_1^0 (1-p) (1-s_h p)}{b_1^0 (1-p) (1-s_h p)+b_2^0p}.
\end{equation}
\label{prop:SFconvergence}
\end{proposition}

\paragraph{Interpretation.} Proposition \ref{prop:SFconvergence} is a rigorous result showing that a single equation on the proportion of {\itshape Wolbachia}-carrying mosquitoes (equation \eqref{eq:p}) is a fair approximation of the time dynamics induced by the model with two populations \eqref{sys:general}, provided that the fecundity is large.

\begin{remark}
It is notable that the function $[0,\eps_0]\ni\eps \mapsto \frac{1+\eps Z(p) - \sqrt{(1-\eps Z(p))^2 + 4 \eps M / (K a(p))}}{2 \eps}$ used to define $n_-$ in the statement of Proposition \ref{prop:SFconvergence} above converges to the finite (and bounded in $p \in [0, 1]$) value $Z(p)-M/(K a(p))$ as $\eps \to 0$. Therefore $n_-$ is uniformly bounded for $\eps \in [0, \eps_0]$.
\end{remark}

\begin{figure}[!ht]
\begin{center}
\includegraphics[width=8cm]{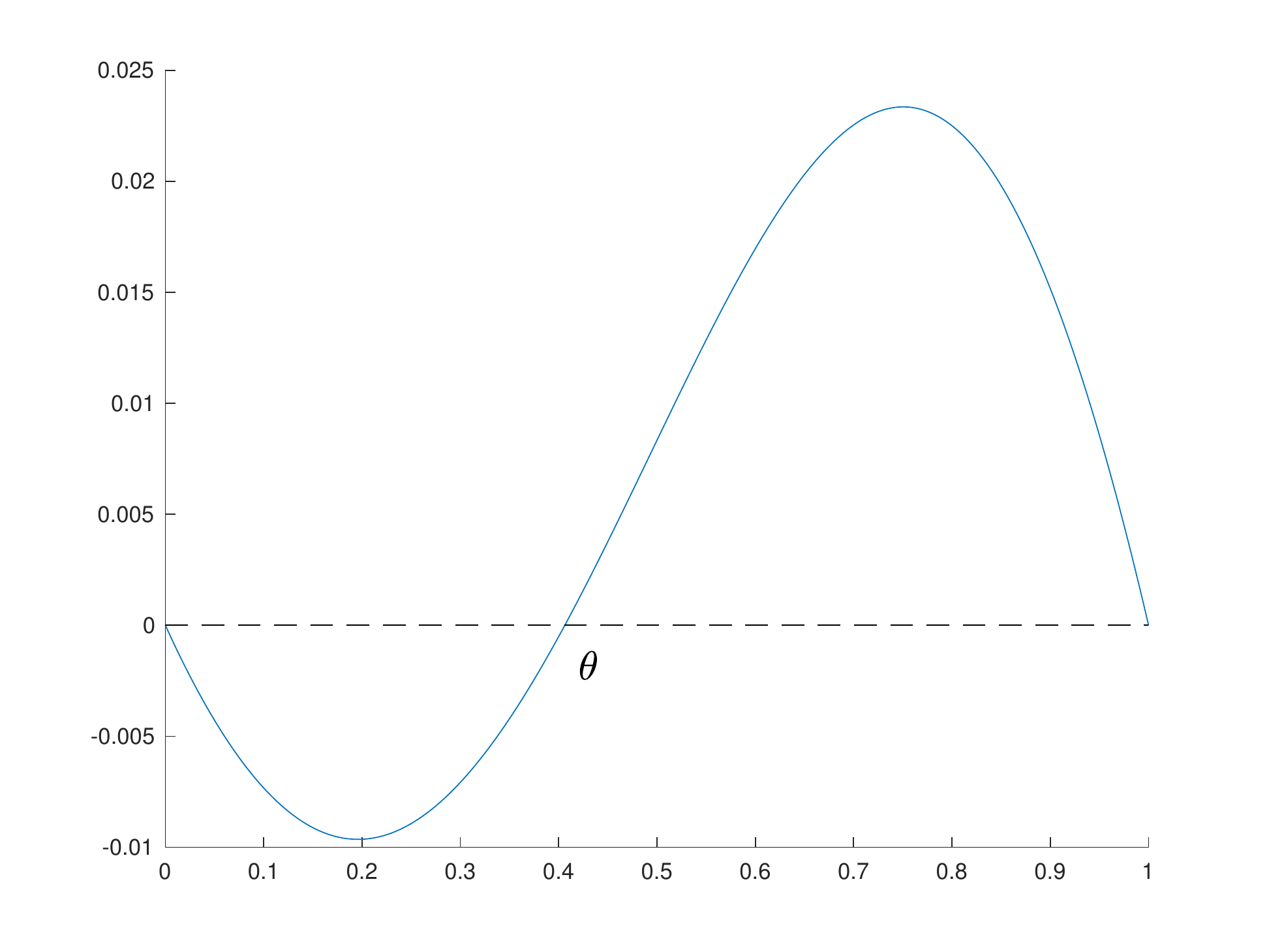}
\caption{Equation \eqref{eq:p} is the form $ \frac{dp}{dt}  = f(p) + u g(p),$ with $f$ of bistable type (see Footnote \ref{footnote.bistable}). Plot of the right-hand side function $f$ with the same parameters values as in Figure \ref{PPsyst}.\label{Fig.bistable}}
\end{center}
\end{figure}

\begin{proof}
System \eqref{sys:general} reads
\begin{equation}\label{eqn1n2}
\begin{cases}
\displaystyle\frac{d n_1^{\eps}}{dt} =& b_1^0 n_1^{\eps} (1-s_h p^{\eps}) n^{\eps} - d_1 n_1^{\eps}   
\\[10pt]
\displaystyle\frac{d n_2^{\eps}}{dt} =& b_2^0 n_2^{\eps} n^{\eps} -d_2 n_2^{\eps} + u^{\eps}.
\end{cases}
\end{equation}
Hence, the resulting system \eqref{syst.nepspeps} on $(n^{\eps},p^{\eps})$ in Proposition \ref{prop:SFconvergence} is obtained from straightforward computations.
 
Let us now provide {\itshape a priori} bounds on $(n^{\eps}(t), p^{\eps}(t))$ (uniform in $\eps \leq \eps_0$, for all $t \geq 0$). Note that $0 \leq p^{\eps} \leq 1$ is an easy consequence of the Cauchy-Lipschitz theorem since $p^{\eps} = 0$ and $p^{\eps} = 1$ are respectively sub- and super-solutions.

We infer that the right-hand side of the equation on $n^{\eps}$ in \eqref{syst.nepspeps} is bounded from below by
\[
 a(p) (1 - \eps n) (Z(p) - n) - \frac{M}{K},
\]
which is positive as soon as $n$ is smaller than the smallest root of this second order polynomial in $n$ given by
\[
 \frac{a(p) (1 + \eps Z(p)) - a(p) \sqrt{(1 - \eps Z(p))^2 + 4 \eps M / (K a(p))}}{2 a(p) \eps}.
\]
Moreover, the right-hand side of the equation on $n^{\eps}$ in \eqref{syst.nepspeps} is bounded from above by
\[
 a(p) (1 - \eps n) (Z(p) - n),
\]
which is negative as soon as $n$ is between $Z(p)$ and $1/\eps$. We then infer the expected uniform estimates on $n^{\eps}$ as soon as $\eps_0$ is small enough.

We are then driven to the slow-fast system \eqref{syst.nepspeps}.
Using the uniform bounds on $n^{\eps}$, $p^{\eps}$, $u^{\eps}$, we infer that the right-hand sides are bounded. Hence, by using the Arzel\`{a}-Ascoli theorem, we get that $(p^{\eps})_{\eps>0}$ converges up to a subfamily uniformly to some function $p$ such that $p(0) = 0$ and $0 \leq p \leq 1$ as $\eps \searrow 0$. Moreover, $dp/dt$ is uniformly bounded since $d p^{\eps}/dt$ is. 

According to Lemma \ref{lem:paris1519} below, the limit $p$ satisfies 
\[
 \frac{d p}{d t} = f(p) + u g(p), \quad p(0) = 0,
\]
with $f$ and $g$ defined by \eqref{deffdefg}. Since the solution to this equation is unique, we finally get the uniform convergence of the whole family $(p^{\eps})_{\eps>0}$ to $p$.
\end{proof}

The two following technical lemmata are used in the proof of Proposition \ref{prop:SFconvergence}.

\begin{lemma}
 Up to a subfamily, the family $(n^{\eps} )_{\eps>0}$ converges weakly to $ Z(p) - \frac{u}{a(p)K}$ as $\eps \searrow 0$ in $(W^{1,1})'$, with $p$ the uniform limit of any subfamily $(p^\eps)_{\eps>0}$.
 \label{lem:auxconv}
\end{lemma}
\begin{proof}
 Let $\phi \in W^{1,1}$ and multiply the differential equation satisfied by $n^{\eps}$ by $\phi$ and integrate by parts over $[0, T]$. We get 
 \[
  \eps [\phi n^{\eps}]_0^T - \eps \int_0^T \frac{d \phi}{d t} n^{\eps} = \int_0^T \phi a (p^{\eps}) (Z (p^{\eps}) - n^{\eps}) - \int_0^T \phi \frac{u^{\eps}}{K}.
 \]
 By weak-star convergence of $u^{\eps}$ in $L^\infty$, uniform convergence of $p^\eps$ in $L^\infty$ and uniform boundedness of $n^{\eps}$ we infer that
 \[
0=  \lim_{\eps \to 0} \int_0^T \phi a (p^{\eps}) (Z (p^{\eps}) - n^{\eps}) - \int_0^T \phi \frac{u}{K},
 \]
leading to the expected result.
\end{proof}

\begin{lemma}\label{lem:paris1519}
Up to a subfamily, $(p^{\eps})_{\eps>0}$ converges uniformly to $p$ solving the ordinary differential equation
 \[
  \frac{d p}{d t} = f(p)+ug(p), \quad p(0) = 0,
 \]
with $f$ and $g$ defined by \eqref{deffdefg}.
\end{lemma}
\begin{proof}
Let us first recast the equation on $p^\eps$ in system \eqref{syst.nepspeps} under the form
 \[
  \frac{d p^{\eps}}{d t} = \beta_{\eps} (n^{\eps}, p^{\eps}, u^{\eps}), \quad p^{\eps} (0) = 0,
 \]
 with
 \[
  \beta_{\eps} (n, p, u)=p (1 - p) \big( n (b_2^0 - b_1^0 (1 - s_h p)) + d_1 - d_2 \big)+ \frac{u(1-p)}{K (1 - \eps n)}
 \]
 so that we easily infer (with obvious notations) that $\beta_{\eps} \to \beta$ as $\eps \searrow 0$ with  $\beta_{\eps} (n, p, u) = n \widehat{\beta} (p) + \beta_0 (p) + u \widetilde{\beta}^{\eps}(n,p)$, $\beta (n, p, u) = n \widehat{\beta} (p) + \beta_0 (p) + u \widetilde{\beta} (p)$, $ \widehat{\beta} (p) =p (1 - p) (b_2^0 - b_1^0 (1 - s_h p))  $, $\beta_0(p)=( d_1 - d_2 )p (1 - p) $ and $\widetilde{\beta} (p)=\frac{1-p}{K (1 - \eps n)}$.
 
Using the previous considerations (and in particular the uniform boundedness of $p^\eps$, $n^\eps$ and $u^{\eps}$), we deduce that $p$ is in fact Lipschitz-continuous, and that $\widehat{\beta}$ and $\widetilde{\beta}$ are continuous on $[0, 1]$.

Now, let us show that $p$ satisfies the limit equation in a weak sense.
Let $\phi \in \mathcal{C}_c^{\infty} (0, T)$. We compute each term separately: the terms in $d p^{\eps} / dt$ and $\beta_0 (p^{\eps})$ converge by uniform convergence of $p^{\eps}$. Therefore, we have
\[
 \int_0^T\phi n^{\eps} \widehat{\beta} (p^{\eps}) = \underbrace{\int_0^T\phi n^{\eps} \widehat{\beta} (p)}_{\to \int_0^T\phi n \widehat{\beta} (p) } + \underbrace{\int_0^T\phi n^{\eps} \big( \widehat{\beta} (p^{\eps}) - \widehat{\beta} (p) \big)}_{\lvert \cdot \rvert \leq \lVert n^{\eps} \rVert_{\infty} o (1) }.
\]
and
\[
 \int_0^T\phi u^{\eps} \widetilde{\beta}^{\eps} (n^{\eps}, p^{\eps}) = \underbrace{\int_0^T\phi u^{\eps} \widetilde{\beta} (p)}_{\to \int_0^T\phi u \widetilde{\beta} (p) } + \underbrace{\int_0^T\phi u^{\eps} \big( \widetilde{\beta}^{\eps} (n^{\eps}, p^{\eps}) - \widetilde{\beta} (p) \big)}_{\lvert \cdot \rvert \leq M o(1) }.
\]
by using simultaneously the weak convergence properties of $(u^\eps)_{\eps>0}$ and $(n^\eps)_{\eps>0}$ (see Lemma \ref{lem:auxconv}) as well as the aforementioned convergence of $\beta_\eps$ to $\beta$.
Here, it is crucial that the limit $\widetilde{\beta}$ does not depend on $n$ but merely on $p$, and we rely on the uniform estimate on $n^{\eps}$.

Finally, a standard argument yields that $p$ must satisfy the equation in a strong sense since it is Lipschitz-continuous.
\end{proof}

We are now in position to determine the asymptotic behavior of the solutions of Problem \eqref{prob:full} as $\eps\searrow 0$, in the case where \eqref{asymp.b1b2} is assumed.

We have already observed that the invasion equilibrium $(0,n_2^*)$ is in particular changed into $(0, K(1-\eps\frac{d_2}{b_2^0}))$, which converges to $(0,K)$ as $\eps\searrow 0$. By using the result stated in Proposition \ref{prop:SFconvergence}, we formally infer that $(N^\eps(T))_{\eps>0}$ converges to $K$ and $(p^\eps(T))_{\eps>0}$ converges to some limit $p(T)\in [0,1]$ as $\eps\searrow 0$, meaning that $(n_1^\eps(T),n_2^\eps(T)))_{\eps>0}$ converges to $(K(1-p(T)),Kp(T))$. It follows that $J^\eps(u)$ converges, as $\eps \searrow 0$ to
\[
\frac{K^2}{2}(1-p(T))^2+\frac{K^2}{2}(1-p(T))^2=K^2(1-p(T))^2,
\]
where $p$ denotes the solution of \eqref{eq:p}.

This leads to introduce an asymptotic version $J^0$ of the cost function $J^\eps$ given by
\begin{equation}\label{J0}
J^0(u) = K^2(1-p(T))^2,
\end{equation}
as well as an asymptotic version of Problem \eqref{prob:full} reading
\begin{equation}\label{prob:reduced}
\boxed{
\inf_{u\in \mathcal{U}_{T,C,M}} (1-p(T))^2,
}
\tag{$\mathcal{P}_{\text{reduced}}$}
\end{equation}
where $p$ solves \eqref{eq:p} and $\mathcal{U}_{T,C,M}$ is defined by \eqref{def.UTCM}.

\medskip

In Section \ref{sec.analysisPfull}, we will analyze the connections between Problem \eqref{prob:full} and Problem \eqref{prob:reduced}, by providing a description of minimizers and highlighting good convergence properties as $\eps \searrow 0$.

\section{Analysis of Problem \texorpdfstring{\eqref{prob:full}}{Pfull} and numerics}\label{sec.analysisPfull}
\subsection{Description of minimizers}\label{sec.descript}
This section is devoted to the analysis of Problems \eqref{prob:full} and \eqref{prob:reduced}. It mainly contains two results:
\begin{itemize}
\item In Prop. \ref{prop:convergence}, we state a $\Gamma$-convergence type result relating the asymptotic behavior of the solutions of Problem \eqref{prob:full} to the ones of Problem \eqref{prob:reduced}. We also investigate existence issues for these problems.
\item In Theorem \ref{thm:reducedprob}, we completely describe the solutions of Problem \eqref{prob:reduced}.
\end{itemize}


\begin{definition}[$\Gamma$-convergence, \cite{Braides.Gamma}] 
One says that $J^{\eps}$ $\Gamma$-converges to $J^0$ if for $u \in \mathcal{U}_{T, C, M}$ and $(u^{\eps})_{\eps>0}$ converging weak-star to $u$ in $L^\infty(0,T)$, one has
  \begin{equation}
   \liminf_{\eps \to 0} J^{\eps} (u^\eps) \geq J^0 (u)
   \label{eq:Gammacond1}
  \end{equation}
  and there exists a sequence $(\overline{u}^{\eps})_{\eps}$, with $\overline{u}^{\eps} \rightharpoonup u$, such that
  \begin{equation}
   \limsup_{\eps \to 0} J^{\eps} (\overline{u}^{\eps}) \leq J^0 (u).
   \label{eq:Gammacond2}
  \end{equation}
\end{definition}
To investigate the convergence of minimizers for Problem \eqref{prob:full},  we will use the fundamental theorem of $\Gamma$-convergence (see e.g. \cite[Theorem 2.10]{Braides.Gamma}) stating that, under a $\Gamma$-convergence property and equicoercivity of the considered functional, closure points of the sequence of minimizers are themselves solution of an asymptotic problem.

\begin{proposition}
 Let $T, C, M > 0$ and assume that \eqref{cond:1} and \eqref{cond:2} hold.
 Problem \eqref{prob:full} and Problem  \eqref{prob:reduced} have (at least) a solution.
  
Moreover, let $(u^\eps)_{\eps>0}$ be a family of minimizers for Problem \eqref{prob:full}. Then, one has 
  \[
  \lim_{\eps \searrow 0}\inf_{u\in \mathcal{U}_{T,C,M}} J^\eps(u)=\inf_{u\in \mathcal{U}_{T,C,M}} J^0(u)
  \]
  and any closure point of this family (as $\eps \searrow 0$, for the $L^\infty$-weak star topology) is a solution of Problem \eqref{prob:reduced}.
 \label{prop:convergence}
\end{proposition}

\paragraph{Interpretation.} Proposition \ref{prop:convergence} establishes that the controlled scalar equation \eqref{eq:p} is not only a fair approximation of the time dynamics of the infection frequency $n_2/(n_1+n_2)$ from system~\eqref{sys:general}, but also provides a sound framework for studying optimization problems. Morally, a release protocol defined by solving the simpler problem \eqref{prob:reduced} will be typically good for \eqref{prob:full} as well, provided that the fecundity is large.

We now solve Problem \eqref{prob:reduced} involving $p$, the solution to \eqref{eq:prop}. In other words
$$
 \frac{dp}{dt}  = f(p) + u g(p),
$$

Moreover, in what follows, we assume that \eqref{cond:2} is satisfied and will mainly use structural properties of $f$ and $g$. Namely they are $C^1$ functions on $[0,1]$ such that $g>0$ on $[0,1)$, $g(1)=0$, and $f$ is a bistable function (see Footnote \ref{footnote.bistable}). We denote by $\theta$ the unique real number satisfying
\[
f(\theta)=0\qquad\text{and}\qquad \theta\in (0,1),
\]
where $f$ is given by \eqref{deffdefg}, in other words,
\begin{equation}\label{def.theta}
\theta=\frac{1}{s_h}\left(1-\frac{d_1b_2^0}{d_2b_1^0}\right).
\end{equation}

\begin{proof}

Let us first investigate the existence of solutions for Problem  \eqref{prob:full} under the assumption \eqref{asymp.b1b2}. 

Fix $\eps>0$ and consider $(u_n^\eps)_{n\in\N}$ a minimizing sequence. According to the Banach-Alaoglu Bourbaki theorem, the set $ \mathcal{U}_{T,C,M}$ is compact for the weak star topology of $L^\infty(0,T)$. Therefore, up to a subsequence, $(u_n^\eps)_{n\in\N}$ converges to some element $u^\eps\in  \mathcal{U}_{T,C,M}$. Let us use the same notation to denote $(u_n^\eps)_{n\in\N}$ and any converging subsequence (with a slight abuse of notation).

An immediate adaptation of the proof of Proposition \ref{prop:SFconvergence} yields successively that $(\nn^\eps_n)_{n\in\N}$ (the sequence of solutions $\nn^\eps_n$ of System \eqref{sys:general} corresponding to $u=u_n^\eps$) is uniformly bounded and converges uniformly to some limit $\nn^\eps$ as $n\to +\infty$, which corresponds to the solution of System \eqref{sys:general} with $u=u^\eps$. We then infer that $(J^\eps(u_n^\eps))_{n\in\N}$ converges to  $J^\eps(u^\eps)$ and the conclusion follows.

To prove the convergence of minimizers as $\eps \searrow 0$ and the existence of solutions for Problem \eqref{prob:reduced}, we will show that $J^{\eps}$ $\Gamma$-converges to $J^0$ as $\eps \to 0$, and conclude by using the fundamental theorem of $\Gamma$-convergence (\cite[Theorem 2.10]{Braides.Gamma}).

With the expressions in \eqref{def.Nepspeps}, we have $n_2^\eps = K p^\eps(1-\eps n^\eps)$ and $n_1^\eps = K (1-p^\eps)(1-\eps n^\eps)$.
  Injecting into the expression of $J^\eps$ in \eqref{Jueps} we obtain
  \[
  J^{\eps} (u^{\eps}) = \frac{K^2}{2} \Big( (1-\eps n^{\eps}(T))^2 (1 - p^{\eps}(T))^2 + \Big[\big(1 - p^{\eps} (T) - \eps(\frac{d_2}{b_2^0}-p^{\eps}(T) n^{\eps}(T)) \big)_+\Big]^2 \Big).
  \]
  By the uniform estimates on $n^{\eps}, p^{\eps}$ provided in Proposition \ref{prop:SFconvergence}, we get that $d p^{\eps} / dt$ is uniformly bounded in $\eps$ on $[0, T]$, and thus by Arzel\`{a}-Ascoli theorem up to extraction $p^{\eps}$ converges uniformly to some $p$. 
{Using also the uniform bounds on $p^\eps$ and $n^\eps$, we can pass to the limit $\eps\to 0$ in the right hand side of the latter equality and get $J^0(p)$, where $J^0$ is defined in \eqref{J0}.}
  
  {In the particular case where the sequence $(\overline{u}^{\eps})_\eps$ is independent of $\eps$, i.e. $\overline{u}^{\eps} = u$,} we get that $\lim_{\eps \to 0} J^{\eps} (u) = J^0 (u)$, which implies \eqref{eq:Gammacond2}. Indeed, according to Proposition \ref{prop:SFconvergence}, the limit $p$ is unique and solves precisely
  \[
   \frac{dp}{dt} = f(p) + u g(p), \quad p(0) = 0.
  \]

Note that Proposition \ref{prop:SFconvergence} proves in fact the stronger result that $J^{\eps} (u^{\eps})$ converges to $J^0 (u)$ as $\eps$ goes to $0$, whence \eqref{eq:Gammacond1}.
\end{proof}

\begin{theorem}
 Let $T$, $C$, $M$  be three positive numbers and assume that $T > C/M$ (in other words that the horizon of time is large enough). 
 Any solution $u$ to \eqref{prob:reduced} satisfies $\int_0^T u^*(t) dt = C$ and is bang-bang ({\itshape i.e.} equal a.e. to 0 or $M$). 
 
 If $\displaystyle M \leq \max_{p \in [0, \theta]} \left(-\frac{f(p)}{g(p)}\right)$ then the unique solution to \eqref{prob:reduced} is given by $M \mathds{1}_{[T-C/M,T]}$.
 
 Otherwise, defining
  \begin{equation}
  C^* (M) = \int_0^{\theta} \frac{M dp}{f(p) + M g(p)},
  \label{expr:Cstar}
 \end{equation}
one has
 \begin{itemize}
  \item if $C < C^*(M)$ then the solution to \eqref{prob:reduced} is unique and equal to $u^* = M \mathds{1}_{[T-C/M, T]}$. In this case $J^0 (u^*) > (1 - \theta)^2$;
  \item if $C > C^*(M)$ then the solution to \eqref{prob:reduced} is unique and equal to $u^* = M \mathds{1}_{[0, C/M]}$. In this case $J^0 (u^*) < (1 - \theta)^2$;
  \item if $C = C^*(M)$ then there is a continuum of solutions to \eqref{prob:reduced} given by $u^*_{\lambda} = M \mathds{1}_{[\lambda, \lambda + C/M]}$ for $\lambda \in [0, T-C/M]$, with $J^0(u^*_{\lambda}) = (1 -\theta)^2$,
 \end{itemize}
 where $\theta$ is given by \eqref{def.theta}.
\label{thm:reducedprob}
\end{theorem}
Theorem \ref{thm:reducedprob} is illustrated on Fig. \ref{fig.illustheo1}.

\paragraph{Interpretation.} Theorem \ref{thm:reducedprob} implies that the best release protocol in the framework of the frequency model \eqref{eq:p} consists in a single release phase, either at the beginning of the time frame or at the end.
\red{This result may be interpreted as follows : If the amount of mosquitoes available is enough to cross the threshold $\theta$, then it is preferable to make the maximum effort at the beginning of the time protocol, since it is clear from the differential equation \eqref{eq:p} satisfied by the frequency $p$, that $p$ is increasing whenever $p > \theta$ even when $u=0$ (thus it is interesting to cross this threshold as soon as possible). On the contrary, if the amount of mosquitoes is not enough for $p$ to reach the threshold $\theta$, then $p < \theta$ and $p$ is decreasing when $u=0$. In this case, the optimum is achieved acting at the end of the time frame to avoid $p$ to decrease\footnote{Note that under the same assumptions, the same conclusion holds for any problem derived from \eqref{prob:reduced} by changing the Cauchy data $p(0) = 0$ in \eqref{eq:p} by an arbitrary value in $[0, 1)$. Therefore we can state that in general, if the threshold cannot be reached, then it is best to act only at the end of the time frame. This fact ultimately relies on the special variations of $f/g$, see the proof of Lemma \ref{lem:2032mP}.}.
To the best our knowledge, such a ``jump'' of optimizers is not standard in optimal control theory. Nevertheless, we point out reference \cite{LLNP} where a close phenomenon is observed for issue of minimizing with respect to the domain the principal eigenvalue of an elliptic operator with Robin boundary conditions of the kind $\partial_n u+\beta u=0$, where $\beta $ is a non-negative constant. Indeed, in this problem, two distinct families of optimizers can arise, depending of whether $\beta$ is larger ou smaller than a threshold value $\beta^*>0$.}

To what extent must this strategy be adapted when $\eps > 0$ is small but nonzero ({\itshape i.e.} in the real situation where fecundity is large but finite)? Numerical results in Section \ref{sec:num} begin to answer this challenging question.\\

The proof relies on several intermediary lemmas which we state and prove below.

\medskip

\red{In view of stating the (necessary) first order optimality conditions for Problem \eqref{prob:reduced}, let us compute the derivative of $J^0$. Let us introduce the adjoint state $q$ defined by
\begin{equation}\label{eq:adjoint}
- \dot{q} = \big( f'(p) + u g'(p) \big) q\text{ on }(0,T), \qquad q(T) = -2 (1-p(T)).
\end{equation}
Standard arguments yield existence and uniqueness of a solution for System  \eqref{eq:adjoint}.
Moreover, since $0<p(\cdot)<1$, we deduce that $q(\cdot)<0$ on $[0, T]$.}

\red{\begin{lemma}\label{lem.GDJ0}
Let $u\in \mathcal{U}_{T,C,M}$. Then, for every admissible perturbation\footnote{More precisely, we call ``admissible perturbation'' any element of the tangent cone $\mathcal{T}_{u,\mathcal{U}_{T,C,M}}$ to the set $\mathcal{U}_{T,C,M}$ at $u$. The cone $\mathcal{T}_{u,\mathcal{U}_{T,C,M}}$ is the set of functions $h\in L^\infty(0,T)$ such that, for any sequence of positive real numbers $\varepsilon_n$ decreasing to $0$, there exists a sequence of functions $h_n\in L^\infty(0,T)$ converging to $h$ as $n\rightarrow +\infty$, and $u+\varepsilon_nh_n\in\mathcal{U}_{T,C,M}$ for every $n\in\N$ (see e.g. \cite{MR1367820}).\label{footnote:cone}
} $h$, the G\^ateaux-derivative of $J^0$ at $u$ in the direction $h$ reads
\[
\langle dJ^0(u),h \rangle=\int_0^T h(t) q(t) g(p(t)) \, dt.
\]
where $p$ and $q$ denote respectively the solutions to the problem \eqref{eq:p} and the adjoint equation \eqref{eq:adjoint}.
\end{lemma}}
\red{\begin{proof}
Let $h$ be an admissible perturbation of $u$ (see Footnote \ref{footnote:cone}). The G\^ateaux-differentiability of $J^0$ is standard and follows from the differentiability of the mapping $\mathcal{U}_{T,C,M}\ni u\mapsto p$, where $p$ denotes the unique solution of \eqref{eq:p}, itself deriving from the application of the implicit function theorem combined with variational arguments. \\
Let us then compute the G\^ateaux-derivative of $J^0$ at $u$ in the direction $h$, defined by
\[
\langle dJ^0(u),h \rangle = \lim_{\varepsilon\to 0} \frac{J^0(u+\varepsilon h)-J^0(u)}{\varepsilon}.
\]
Let $p$ be the solution to Eq. \eqref{eq:prop} and let us introduce $\delta p$, the G\^ateaux-differential of $p$ at $u$ in the direction $h$. Straightforward computations yield that $\delta p$ solves the linearized problem to \eqref{eq:prop},
\[\label{eq:dotp}
\dot{\delta p}(t) = f'(p) \delta p + u g'(p) \delta p + h g(p),
\qquad \delta p(0) = 0.
\]
Furthermore, differentiating the criterion $J^0$ yields $\langle dJ^0(u),h \rangle =-\delta p(T)(1-p(T))$. Let us multiply Eq. \eqref{eq:dotp} by $q$ and then integrate by parts on $(0,T)$. One gets
\begin{eqnarray*}
-\int_0^T \dot{q}(t)\delta p(t)\, dt=-\delta p(T)q(T)+\int_0^T(f'(t)+ug'(t))\delta p(t)q(t)\, dt+\int_0^T h(t)g(p(t))q(t)\, dt,
\end{eqnarray*}
which reduces to
$$
\delta p(T)q(T)=\int_0^T h(t)g(p(t))q(t)\, dt.
$$
By using the terminal condition on $q$, we obtain
\[
\langle dJ^0(u),h \rangle = - 2 (1-p(T)) \delta p(T) = q(T) \delta p(T),
\]
where $q$ is the solution to the adjoint equation \eqref{eq:adjoint}.
Then, we compute
\begin{align*}
0 = \int_0^T \delta p (\dot{q} + f(p) q + u g'(p) q) \,dt 
= \delta p(T) q(T) - \delta p(0) q(0) - \int_0^T h(t) q(t) g(p(t))\,dt,
\end{align*}
and it follows that
\[
\langle dJ^0(u),h \rangle = \int_0^T h(t) q(t) g(p(t)) \,dt.
\]
\end{proof}}

\red{We now prove that the $L^1$ constraint on the control $u$ is saturated.
\begin{lemma}\label{lem:conssatur}
If $u^*$ solves the optimization problem \eqref{prob:reduced}, then $\displaystyle \int_0^T u^*(t)\,dt = \min(C,TM)$.
\end{lemma}
\begin{proof}
Let us argue by contradiction, assuming the existence of a positive number $\kappa$ and a set $\mathcal{I}\subset (0,T)$ of positive Lebesgue measure such that $0 \leq u^*\leq M-\kappa$ on $\mathcal{I}$ and $\int_0^T u^*(t)\, dt <C$. Then, there exists a positive function $h$ in $L^\infty(0,T)$ such that $u^*+h $ belongs to $\mathcal{U}_{T,C,M}$. Moreover, according to Lemma \ref{lem.GDJ0}, one has
$$
\lim_{\varepsilon \searrow 0}\frac{J^0(u^*+\varepsilon h)-J^0(u^*)}{\varepsilon}=\int_0^T h(t) q(t) g(p(t)) \, dt <0,
$$
where $p$ and $q$ denote here the solutions of Eq. \eqref{eq:p} and \eqref{eq:adjoint} associated to $u=u^*$, by using the positivity of $h$ and $g$, and the negativity of the adjoint state $q$. This is in contradiction with the optimality of $u^*$ and therefore, either $u^*=M$ a.e. (which arises if, and only if $TM\leq C$) or $\int_0^T u^*(t)\, dt=C$. The expected conclusion follows.
\end{proof} }

\begin{lemma}
Let $u \in \mathcal{U}_{T,C,M}$ be a solution of Problem \eqref{prob:reduced}. Define the {\itshape switching function} $w$ by $w(t) = g(p(t)) q(t)$ for all $t\in [0,T]$, where $p$ and $q$ denote here the solutions of Eq. \eqref{eq:p} and \eqref{eq:adjoint} associated to $u=u^*$. There exists $\Lambda < 0$ such that
 \begin{itemize}
  \item $u(t) = M \Rightarrow w(t) \leq \Lambda$,
  \item $0 < u(t) < M \Rightarrow w(t) = \Lambda$,
  \item $u(t) = 0 \Rightarrow w(t) \geq \Lambda$,
 \end{itemize}
each equality being understood up to a zero Lebesgue-measure set.
  \label{lem:fromord1}
\end{lemma}
\begin{proof}
 Introduce the Lagrangian function~$\mathcal{L}$ associated to Problem \eqref{prob:reduced}, defined by
\[
	\mathcal{L} : \mathcal{U}_{T,C,M}\times \R\ni (u,\Lambda) \mapsto J^0(u)  -\Lambda \left(\int_0^T u(t)\; dt -C \right).
\]
Standard arguments enable to show the existence of a Lagrange multiplier $\Lambda$ such that $(u,\Lambda)$ is a saddle-point of the Lagrangian functional $\mathcal{L}$. Moreover, according to Lemma \ref{lem:conssatur} and since $T>C/M$, we have necessarily $\int_0^T u =C$. 

Let $x_0$ be a density-one point of  $\{u=M\}$. Let $(G_{n})_{n\in \N}$ be a sequence of measurable subsets with $G_{n}$ included in $\{u=M\}$ and containing $x_0$. Let us consider $h=\mathds{1}_{G_{n}}$ and notice that $u -\eta h$ belongs to $\mathcal{U}_{T,C,M}$ whenever $\eta$ is small enough. Writing 
\[
\mathcal{L}(u - \eta h,\Lambda )\geq \mathcal{L}(u ,\Lambda ),
\] 
dividing this inequality by $\eta$ and letting $\eta$ go to 0, it follows that 
\[
-\langle d J^0(u),h \rangle +\Lambda \int_0^T h(t)\, dt\geq 0,
\]
or equivalently that
\[
- \int_{G_{n}} q(t)g(p(t))+\Lambda |G_{n}|\geq 0,
\]
according to Lemma \ref{lem.GDJ0}. Dividing this inequality by $|G_{n}|$ and letting $G_{n}$ shrink to $\{x_0\}$ as $n\to +\infty$ shows the first point of Lemma \ref{lem:fromord1}, according to the Lebesgue Density Theorem. The proof of the third point is similar, and consists in considering perturbations of the form $u+\eta h$ where $h$ denotes a positive admissible perturbation of $u$ supported in $\{u(t)=0\}$. Finally, the proof of the second point follows the same lines, by considering bilateral perturbations of the form $u\pm \eta h$ where $h$ denotes an admissible perturbation of $u$ supported in $\{0<u(t)<M\}$.
\end{proof}

 \begin{lemma}\label{lem:2032mP}
Let $u \in \mathcal{U}_{T,C,M}$ be a solution of Problem \eqref{prob:reduced} and let $p$ and $q$ be the solutions of Eq. \eqref{eq:p} and \eqref{eq:adjoint} associated to $u$. 

Under the assumption \eqref{cond:2}:
\begin{itemize}
\item if $M \leq \max_{[0,1]} - f/g$, then $u$ is bang-bang.
\item if $M > \max_{[0,1]} - f/g$, $u$ is either bang-bang or constant and equal to $-f(p^*) / g(p^*)$ (the latter case may occur only if $C = - T f(p^*) / g(p^*)$).
\end{itemize}
 \end{lemma}
 \begin{proof}
 Similarly to the statement of Lemma \ref{lem:fromord1}, let us introduce the function $w$ defined by 
 $$
 w(t)=q(t) g\circ p(t).
 $$
for all $t\in [0,T]$. In optimal control theory, $w$ is the so-called {\it switching function}.
 
 Let us differentiate $w$. We get
 \begin{align*}
  \frac{dw}{dt}(t)  &= q'(t) g (t)+ p' (t)g' (t)q (t)
  \\
  &= (-f'(t) - u(t)g'(t))g(t)q(t) + (f(t)+u(t)g(t))g'(t) q(t) 
  \\
  &= q(t) \big( f(p(t)) g'(p(t)) - f'(p(t)) g(p(t)) \big).
 \end{align*}
Combining this computation with Remark \ref{rk.xi} below yields the existence of a unique real number $p^* \in (0, 1)$ such that $(f/g)' (p^*) = 0$. Now, it follows from Lemma \ref{lem:fromord1} that one has $p'(t)=0$ on $\{0<u<M\}$ (see e.g. \cite[Remark 3.1.10]{henrot-pierre} for a proof of this fact). Plugging this equality in the main equation of \eqref{eq:p} yields that
$$
u(t)=-\frac{f(p^*)}{g(p^*)}\qquad \text{on}\quad \{0<u<M\}.
$$
At this step, we have proved that the optimal control $u$ satisfies $u(t)\in \{0,-f(p^*)/g(p^*),M\}$ for a.e. $t\in (0,T)$.

Let us distinguish between two cases:
\begin{itemize}
\item if $M\leq \max_{\bar p\in [0,\theta]}[-f(\bar p)/g(\bar p)]$ (note that the maximum of $-f/g$ over $[0,\theta]$ is the maximum of $-f/g$ over $(0,1)$), then, one gets immediately that $|\{0<u<M\}|=0$ (else, the necessary optimality conditions would show that $u$ does not belong to the admissible set $\mathcal{U}_{T,C,M}$).
\item else, if $M> \max_{\bar p\in [0,\theta]}[-f(\bar p)/g(\bar p)]$, we claim that
\begin{equation}\label{train:0742}
\{0<u<M\}=\left\{t\in (0,T)\mid p(t)=p^*\right\}=\left\{t\in (0,T)\mid u(t)=-\frac{f(p^*)}{g(p^*)}\right\}.
\end{equation}
Indeed, according to Lemma \ref{lem:fromord1}, one has 
$$
\{0<u<M\}\subset \left\{t\in (0,T)\mid p(t)=p^*\right\}\subset\left\{t\in (0,T)\mid u(t)=-\frac{f(p^*)}{g(p^*)}\right\}.
$$
By using that $u(t)\in \{0,-f(p^*)/g(p^*),M\}$ for a.e. $t\in (0,T)$, if $u(t)=-\frac{f(p^*)}{g(p^*)}$ on a positive measure set, then $u(t)<M$ because of the above assumption on $M$ and furthermore $u(t)>0$ (else, $f(p^*)=0$ which is impossible since $p^*\in (0,\theta)$). Hence, one has $\left\{t\in (0,T)\mid u(t)=-\frac{f(p^*)}{g(p^*)}\right\}\subset \{0<u<M\}$, whence the claim.

Using that $\{0<u<M\}=\{p=p^*\}$ and since $\{t\in (0,T)\mid w(t)=\Lambda\}\subset \left\{t\in (0,T)\mid p(t)=p^*\right\}$, it follows that
$$
\{w=\Lambda\}=\{0<u<M\}, \quad \{u=M\}=\{w<\Lambda\}\quad \text{and}\quad \{u=0\}=\{w>\Lambda\}.
$$

Between $0$ and $1$, $f$ changes sign only once, at $\theta$.
In addition, the switching function $w$ is decreasing if $p(t) < p^*$ and increasing if $p(t) > p^*$, since it is positively proportional to $(f/g)'$, which changes sign only once, and $f/g$ changes sign only once, and is decreasing at $0$, so $(f/g)'$ has the same sign as $p-p^*$.
Notice that necessarily, $\theta \geq p^*$. Indeed, $f/g$ is decreasing on $(0, p^*)$ and equal to $0$ at $0$ and $\theta$.

Let $\bar t$ be a switching point from $u=-\frac{f(p^*)}{g(p^*)}$ to $u=0$. Such a definition makes sense by interpreting $\bar t$ as a well-chosen endpoint of a connected component of $\{u=0\}$, which is an open set. 

According to the considerations above, $p$ must decrease since $p^* < \theta$, so $p(t) < p^*$ at $\bar t$, and therefore $w$ must decrease at $\bar t$. But this is in contradiction with the necessary optimality condition of Lemma \ref{lem:fromord1}. If at $\bar t$ we have $u=M$ then $p$ must increase if $M$ is large enough. Then $w$ must increase, and again this is in contradiction with Lemma~\ref{lem:fromord1}. Hence $|I| = 0$ or $|I| = T$. But $I= [0, T]$ is admissible if and only if $- T f(p^*) / g(p^*) = C$.

\end{itemize}
\end{proof}

At this step, we have shown that any optimal control $u$ is {\it bang-bang} whenever $T>C/M$. It remains to determine optimal configurations among {\it bang-bang} functions, which is the goal of what follows.

Let us define $p_M$ as the solution of 
\[
 \frac{dp_M}{dt} = f(p_M) + M g(p_M), \quad
 p_M(0) = 0.
 \]
Assume that $M > \max_{p\in[0,\theta]} -\frac{f(p)}{g(p)}$.
Then $\frac{dp_M}{dt}  = f(p_M) + M g(p_M) > 0.$
Introduce the function $G_M$ defined by $G'_M(p)= \frac{1}{f(p)+M g(p)}$ and $G_M(0)=0$. Then, $G_M$ is an increasing function and we have
\[
G_M(p_M(t)) = G_M(p_M(t_0)) + t-t_0, \quad \mbox{ and } \quad 
G_M(p_M(C/M)) = \frac{C}{M}.
\]

The use of all these results allows us to prove Theorem \ref{thm:reducedprob}.

\begin{proof}[Proof of Theorem \ref{thm:reducedprob}]
 We split the proof into three cases:
\begin{itemize}
\item {\it Case $p_M(C/M)<\theta$.} This condition is equivalent to $G_M(p_M(C/M)) < G_M(\theta)$ (since $G_M$ is increasing).
By Lemma \ref{lem:2032mP}, the control $u$ is bang-bang and the set where $u=M$ is open, (since from Lemma \ref{lem:fromord1}, it is the set of interval on which $g(p)q<\Lambda$).
Consider that $u$ is given by $u(t)=M \sum_{i\in \N}\mathds{1}_{(t_{2i+1},t_{2i+2})}$, 
where $t_0=0$ and $(t_i)_{i\in \N}$ is a non-decreasing sequence of times in $[0,T]$. 
We denote by $p$ the corresponding solution to \eqref{eq:prop}.

We want to compare with the control $\bar{u} = M \mathds{1}_{[T-C/M, T]}$, for which the corresponding solution to \eqref{eq:prop} is denoted $\bar{p}$.
Then, $\bar{p}(T) = G_M^{-1}(C/M)$.

Let us show that $p(T)<\bar{p}(T)=G_M^{-1}(C/M)$.
We use an induction to prove that for all $i\in \N$, $p(t_{2i}) < G^{-1}_M(C/M)$. 
Indeed, if we assume that for a $i\in\N^*$, we have for every $k\leq i$, $p(t_{2k}) < G_M^{-1}(C/M)<\theta$.
Then, on $(t_{2i},t_{2i+1})$, we solve the equation
\[
\dot{p} = f(p) , \qquad p(t_{2i}) < \theta.
\]
Since $f<0$ on $(0,\theta)$, it implies that $p$ is decreasing on $(t_{2i},t_{2i+1})$, thus $p(t_{2i+1})<p(t_{2i})$.
On $[t_{2i+1},t_{2i+2})$, we have
\[
G_M(p(t_{2i+2})) = G_M(p(t_{2i+1})) + t_{2i+2}-t_{2i+1} <
G_M(p(t_{2i})) + t_{2i+2}-t_{2i+1}.
\]
By induction, we deduce that
\begin{align*}
G_M(p(t_{2i+2})) < & G_M(p(t_{2i-2})) + t_{2i}-t_{2i-1} + t_{2i+2}-t_{2i+1}  \\
< & G(p(t_0)) + t_2-t_1 + \ldots + t_{2i+2}-t_{2i+1} \leq C/M,
\end{align*}
since $p(t_0)=0$ and $\displaystyle \sum_{k=0}^i (t_{2k+2}-t_{2k+1}) \leq \frac{C}{M}$.
We infer
\[
p(t_{2i+2}) < G_M^{-1}(C/M).
\]
This concludes the induction and the proof in this first case.

\item {\itshape Case $p_M(C/M) > \theta$.} 
We use the same strategy and introduce the solution $p$ to \eqref{eq:prop} with 
$u$ given by $u(t)=M \sum_{i\in \N}\mathds{1}_{(t_{2i},t_{2i+1})}$, 
where $(t_i)_{i\in \N}$ is an increasing sequence of time in $[0,T]$. 
We want to compare with the solution $\bar{p}$ for $\bar{u}=M\mathds{1}_{[0,C/M]}$.

We first observe that since $\bar{p}(C/M)=p_M(C/M) > \theta$ and $f>0$ on $(\theta,1)$, we have $\bar{p}$ increasing on $[C/M,T]$ and $\bar{p}(C/M)=G_M^{-1}(C/M)$.
If at time $t_{1}$, we have $p(t_1)<\theta$, then on $(t_1,t_2)$, $p$ is decreasing. Then $p(t_2)<p(t_1)\leq \bar{p}(t_1-t_0)$ and we may prove as above that as long as $p(t_{2i+1})<\theta$, we have $p(t_{2i+2})<\bar{p}\left(\sum_{k=0}^i(t_{2k+1}-t_{2k})\right)$.

As a consequence the solution $p$ associated with the optimal control should satisfy $p(t_1)>\theta$. Then, on $(t_1,T)$ the function $p$ solving \eqref{eq:prop} is increasing, thus on $(t_1,T)$, we have $p>\theta>p^*$. Then the switch function $w$ is increasing. However, we have $w > \Lambda$ on $(t_1,t_2)$ since $u=0$ from Lemma \ref{lem:fromord1}.
Hence, it is not possible to have $u=M$ for larger times.

\item {\itshape In the case where $p_M(C/M) = \theta$.}
In this case, we have $u=M \mathds{1}_{(\lambda,C/M+\lambda)}$ for any $0\leq \lambda \leq T-C/M$. 
Indeed, for such a function, we have $p \equiv 0$ on $[0,\lambda]$ and $p \equiv \theta$ on $[C/M+\tau,T]$. 
By contradiction, assume there is an interval on which $u=0$ between two intervals on which $u=M$, then on this interval $p$ is decreasing, and thus $p$ cannot reach the value $\theta$ at the final time of control, by comparison.
\end{itemize}
\end{proof}

\begin{remark} \label{rk.xi}
It is notable that the proof of Theorem \ref{thm:reducedprob} rests upon a property of the functions involved in Equation  \eqref{eq:p}, namely the existence of a unique $p^* \in (0, 1)$ such that $(f/g)'(p^*) = 0$, and $C \not= -T f(p^*)/g(p^*)$.
Indeed, letting $\xi = \frac{d_1 b_2^0}{d_2 b_1^0}$ we have
\[
 \frac{f}{g}(p) = K d_2 \big( \frac{p }{1-s_h p} \xi - p \big), \quad \big( \frac{f}{g} \big) '(p) = K d_2 \big( \frac{1}{(1-s_h p)^2} \xi - 1 \big).
\]
The roots of the second-order polynomial at the numerator of the right-hand side read 
\[
 p_{\pm} = \frac{1}{s_h} \big( 1 \pm \sqrt{\xi} \big),
\]
so assuming \eqref{cond:2} ({\itshape i.e.} $\xi < 1$) yields 
\[
 p^* = \frac{1}{s_h} \big( 1 - \sqrt{\xi} \big)
\]
(which indeed belongs to $[0,\theta)$ as a consequence of \eqref{cond:2}: from $1 - s_h < \xi < 1$ it follows that $0 < p^ * < (1 - \sqrt{1-s_h})/s_h < 1$ since $s_h \in (0, 1]$).
On the contrary, assuming $d_2 b_1^0 < d_1 b_2^0$ ({\it i.e.} $\xi > 1$) implies that there is no such $p^*$ in $[0, 1]$ (and in this case the control must be bang-bang, as a consequence of Lemma \ref{lem:2032mP} above).
\end{remark}

\begin{figure}[!ht]
\begin{center}
\includegraphics[width=4.8cm]{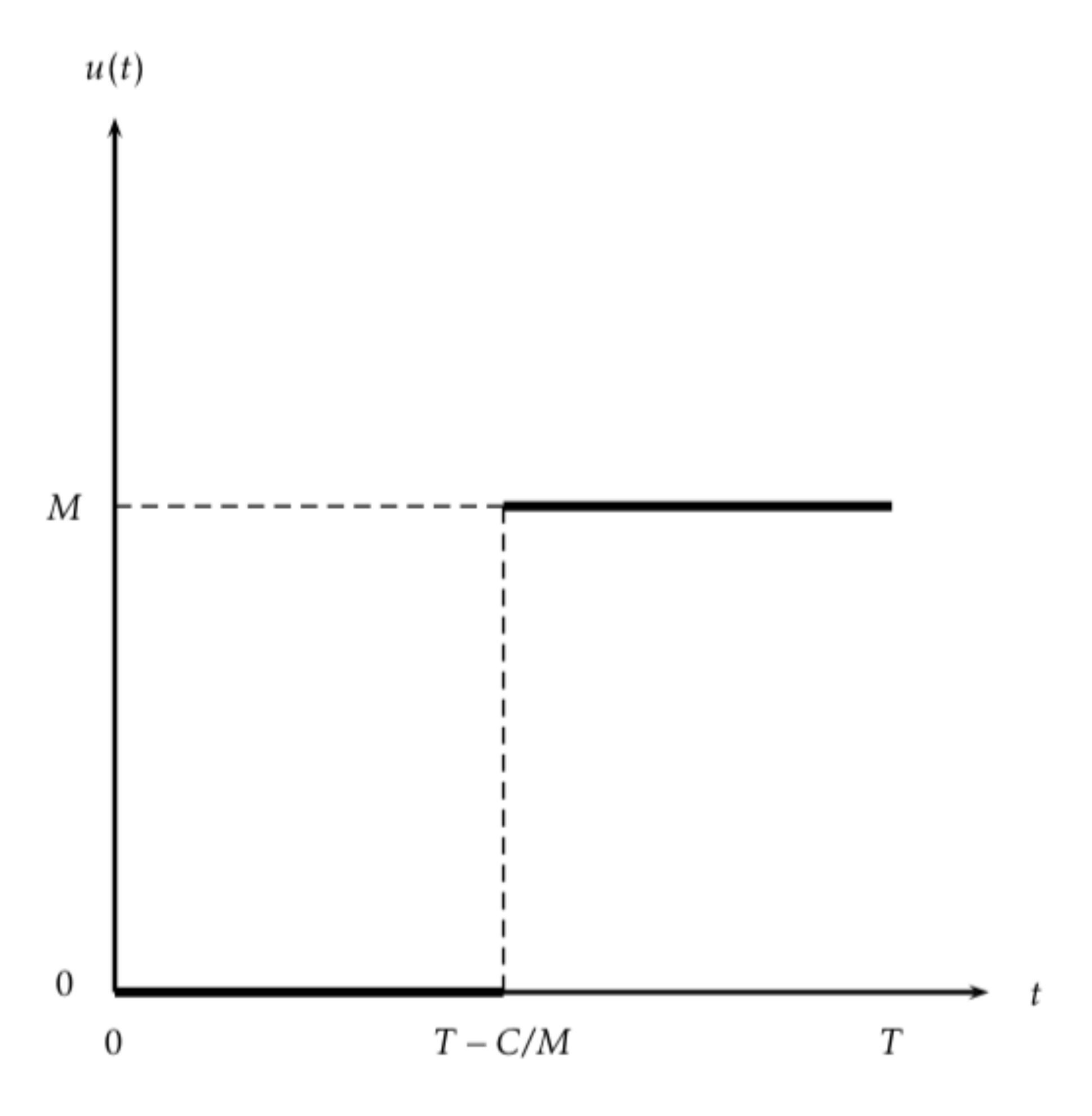}
\includegraphics[width=4.8cm]{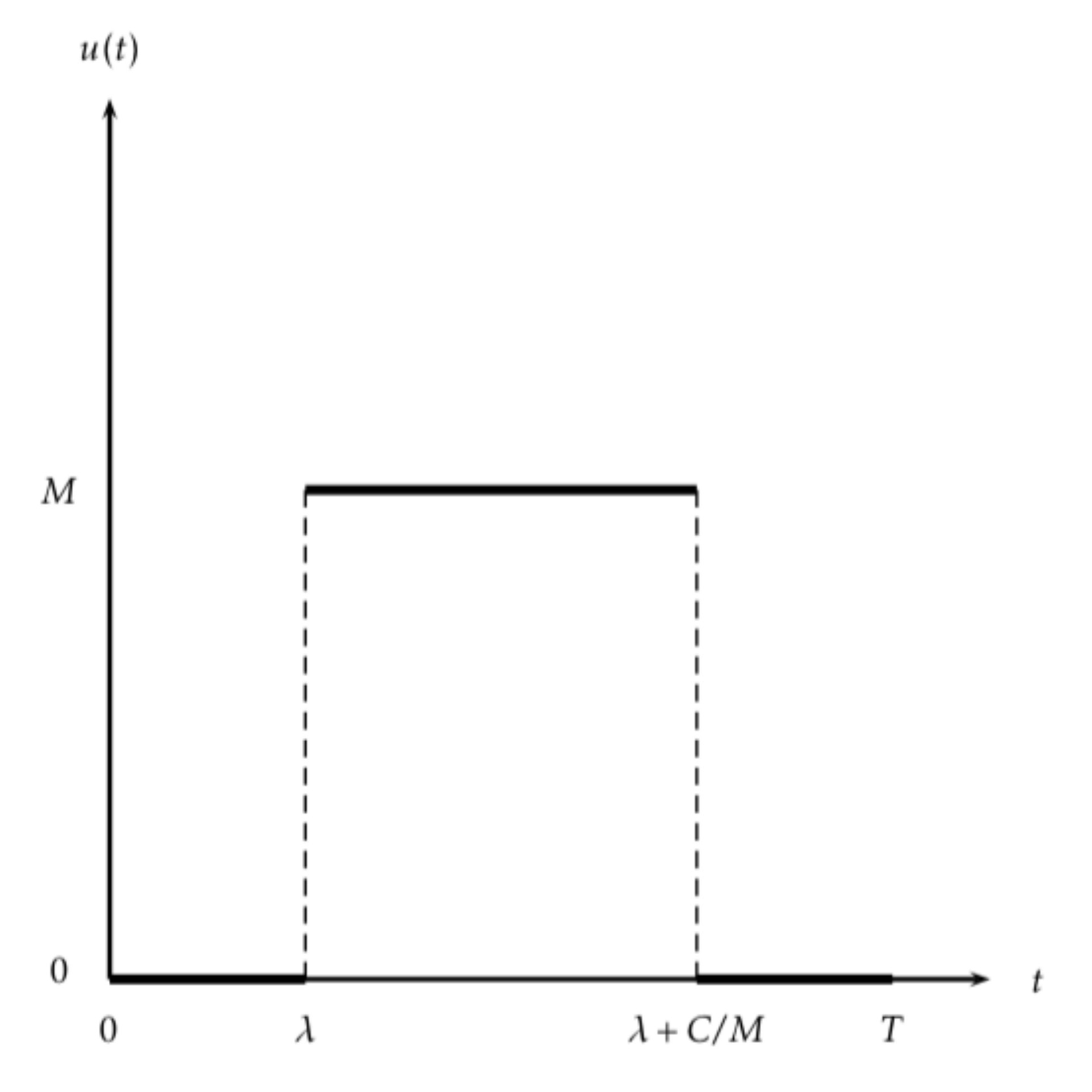}
\includegraphics[width=4.8cm]{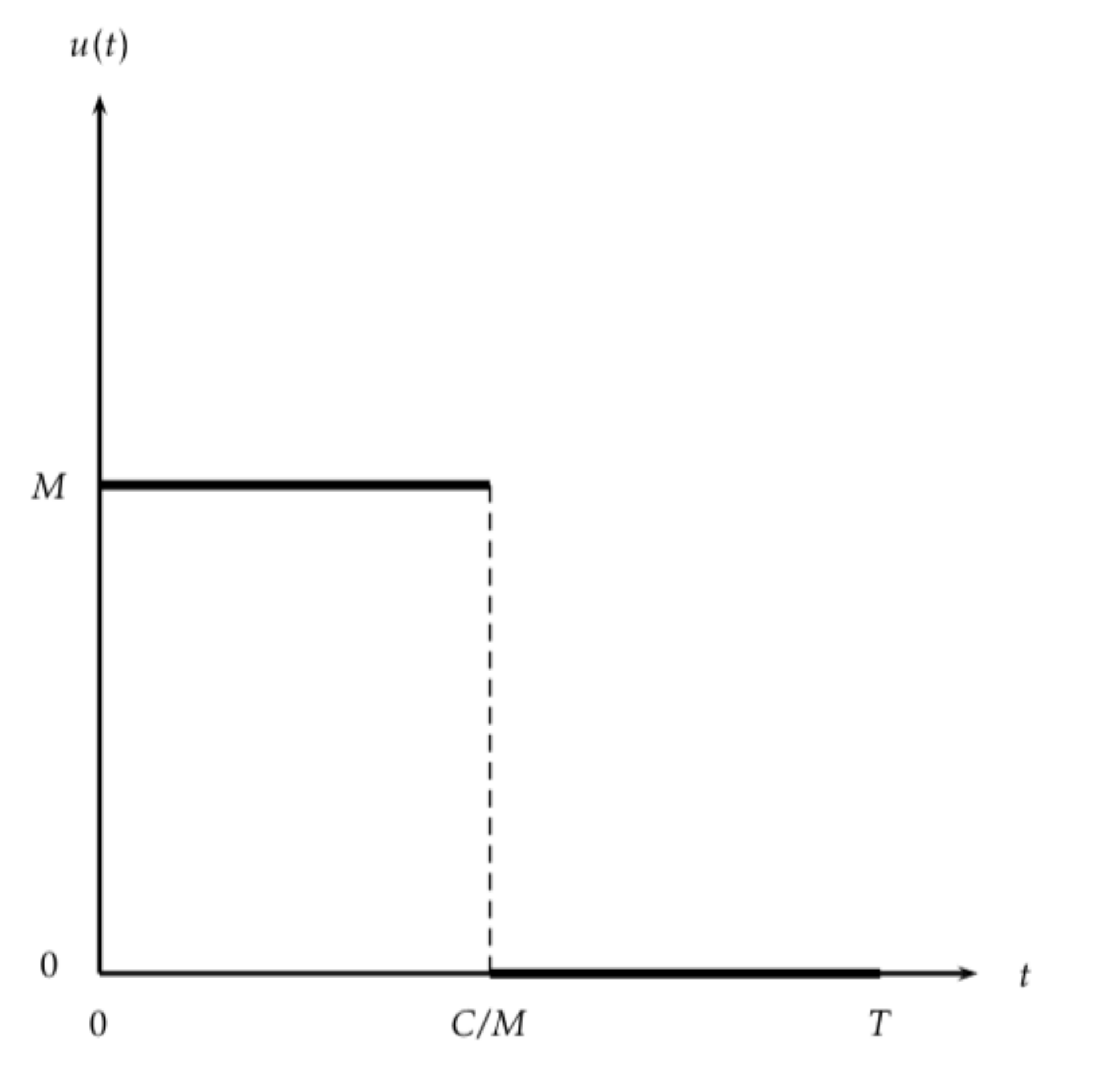}
\caption{Left: solution $u^*$ in the case $M > \max_{p \in [0, \theta]} -f(p)/g(p)$ and $C>C^*(M)$. Middle: one solution $u^*_\lambda$ in the case $M > \max_{p \in [0, \theta]} -f(p)/g(p)$ and $C=C^*(M)$. Right: solution $u^*$ in the case $M \leq \max_{p \in [0, \theta]} -f(p)/g(p)$ or $M > \max_{p \in [0, \theta]} -f(p)/g(p)$ and $C<C^*(M)$.
.\label{fig.illustheo1}}
\end{center}
\end{figure}

From Proposition \ref{prop:convergence} and Theorem \ref{thm:reducedprob}, we provide hereafter a more precise result about the convergence of optimal values for Problem \eqref{prob:full}  as $\eps \searrow 0$. 

\begin{corollary}\label{cor.strongCV}
Let $(u^\eps)_{\eps>0}$ be a family of minimizers for Problem \eqref{prob:full}. Then, $(u^\eps)_{\eps>0}$ converges strongly in $L^1(0,T)$ to a solution of Problem \eqref{prob:reduced} as $\eps\searrow 0$ (which is unique whenever $C\neq C^*(M)$ with the notations of Theorem \ref{thm:reducedprob}).
\end{corollary}
\begin{proof}
According to Proposition \ref{prop:convergence}, we know that $(u^{\eps})_{\eps>0}$ converges weak star in $L^\infty(0,T)$ to a solution of Problem \eqref{prob:reduced}, say $u^*$.

Since $u^*$ is an extremal point of the convex set $\mathcal{U}_{T,C,M}$ to which all elements of the sequence $(u^{\eps})_{\eps}$ belong, it follows from \cite{Barder.Equivalence} that the $L^{\infty}$-weak$*$ convergence (that is here, $L^1$-weak convergence) implies strong convergence in $L^1$, and therefore 
\[
 \lim_{\eps \to 0} \int_0^T \lvert u^{\eps}(t) - u^* (t) \rvert dt  = 0.
\]
Finally, we conclude by observing that, whenever $C\neq C^*(M)$, the solution to Problem  \eqref{prob:reduced} is unique according to Theorem \ref{thm:reducedprob}.
\end{proof}
\subsection{Numerics}
\label{sec:num}
This section is devoted to computing the solution of Problem \eqref{prob:full} and to illustrating the relations with its reduced version \eqref{prob:reduced}.

All the simulations are obtained with a direct method applied to the optimal control problem \eqref{prob:full}, consisting in discretizing System \eqref{sys:general}, the control, and to reduce the optimal control problem to some minimization problem with constraints. 
To this aim, we used the open-source optimization routine from \texttt{IPOPT} (see \cite{Wachter.Implementation}) combined with \texttt{ AMPL} modeling language (see \cite{Fourer.Modeling}). This enables the computation of a local minimizer for a discretized version of \eqref{prob:full}. 

\paragraph{Choice of numerical parameters and methods.}
Populations are normalized by setting $K = 1$, and Table \ref{table:params} yields the values used for the other parameters.
The time-dynamics (the slow-fast system \eqref{syst.nepspeps} depending on $\eps$) are discretized with the Runge-Kutta implicit scheme Lobatto IIIC of order $2$ (two stages).  This scheme is asymptotic preserving in $\eps$ (see \cite{Hairer.Error}) and allows for sound comparison of the simulations across a range of values of this parameter.

We obtain a solution $\nn_{\Delta t} \in  (\R_+)^{2 N_d}$ as well as an approximate local minimizer for the discretized problem \eqref{prob:full}, $\widehat{u}^{\eps, \Delta t} \in [0, M]^{N_d}$. \red{Finally, note that several choices of initialization have been tested (among which constant and random ones).}

\begin{table}[!ht]
\centering
 \begin{tabular}{|c|c|c|c|}
  \hline
  \emph{Category} & \emph{Parameter} & \emph{Name} & \emph{Value or range} \\
  \hline
  Discretization & $\Delta t$ & Time step & $[0.0004,0.0015]$ \\
  \hline
  Singular limit & $1/\eps$ & Birth rates normalization & $[1, 2000]$ \\
  \hline
  \multirow{3}{*}{Optimization} & $T$ & Final time & $10$ \\
  \cline{2-4}
  & $C$ & Maximal release number & $[0.15, 0.75]$ \\
  \cline{2-4}
  & $M$ & Maximal release flux & $10$ \\
  \hline
  \multirow{5}{*}{Biology} & $b_1^0$ & Normalized wild birth rate & $1$ \\
  \cline{2-4}
  & $b_2^0$ & Normalized infected birth rate & $0.9$ \\
  \cline{2-4}
  & $d_1$ & Wild death rate & $0.27$ \\
  \cline{2-4}
  & $d_2$ & Infected death rate & $0.3$ \\
  \cline{2-4}
  & $s_h$ & Cytoplasmic incompatibility level & $0.9$ \\
  \hline
 \end{tabular}
\caption{Parameters for the numerical resolution of \eqref{prob:full}}
\label{table:params}
\end{table}

\paragraph{Results.}
It is convenient to introduce the number of steps in the time discretization $N_d = T/\Delta t$.

For the parameters given in Table \ref{table:params}, we can compute the critical value $C^*(M)$ from Theorem~\ref{thm:reducedprob} numerically: it is close to $0.24$. Therefore we choose three values of $C$ ($0.15$, $0.4$ and $0.75$) both above and below this threshold, so as to get contrasting results.
On Figure \ref{fig:numres} below, solutions of Problem \eqref{prob:full} are computed for these three different values of the integral bound $C$ and for $\eps=1$. We observe that the set $I_M^{\Delta t, \kappa} := \{ k \in \llbracket 1, N_d \rrbracket, \, \widehat{u}^{\eps, \Delta t}_k \geq M - \kappa \}$ (approximating the set $\{u=M\}$), for $\kappa$ small enough, is made of two segments containing either $1$ or $N_d$. Let us denote these two segments $\llbracket 1, k_0 (\Delta t) \rrbracket$ and $\llbracket k_1 (\Delta t), N_d \rrbracket$. 
It seems that 
\begin{itemize}
\item a {\itshape relaxation type} phenomenon may occur for optimal controls meaning that the solution is not {\itshape bang-bang}.
\item the set $I_{\textrm{relax}}^{\Delta t, \kappa} := \{ k \in \llbracket 1, N_d \rrbracket, \, \kappa \leq \widehat{u}^{\eps, \Delta t}_k \leq M - \kappa \}$ (approximating the set $\{0<u<M\}$) seems to be a segment for $\kappa$ small enough.
\item for small values of $C$, $k_0 = 0$, $k_1 = N_d$ and there is replacement failure, suggesting that it is necessary to release a minimal number of infected mosquitoes in order to guarantee population replacement.
\end{itemize}

\begin{figure}[!ht]
\centering 
 \includegraphics[width=.32\textwidth]{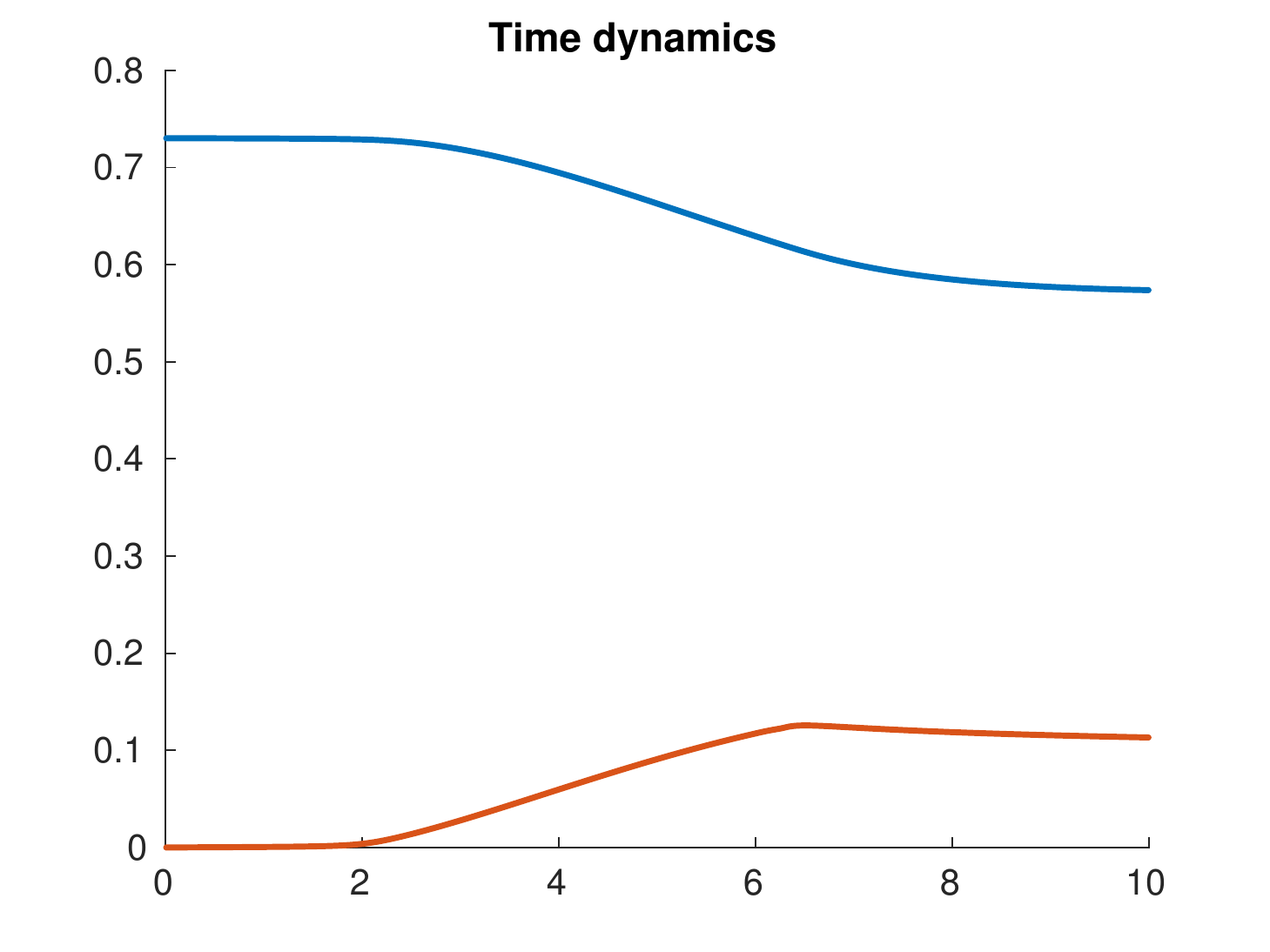}
 \includegraphics[width=.32\textwidth]{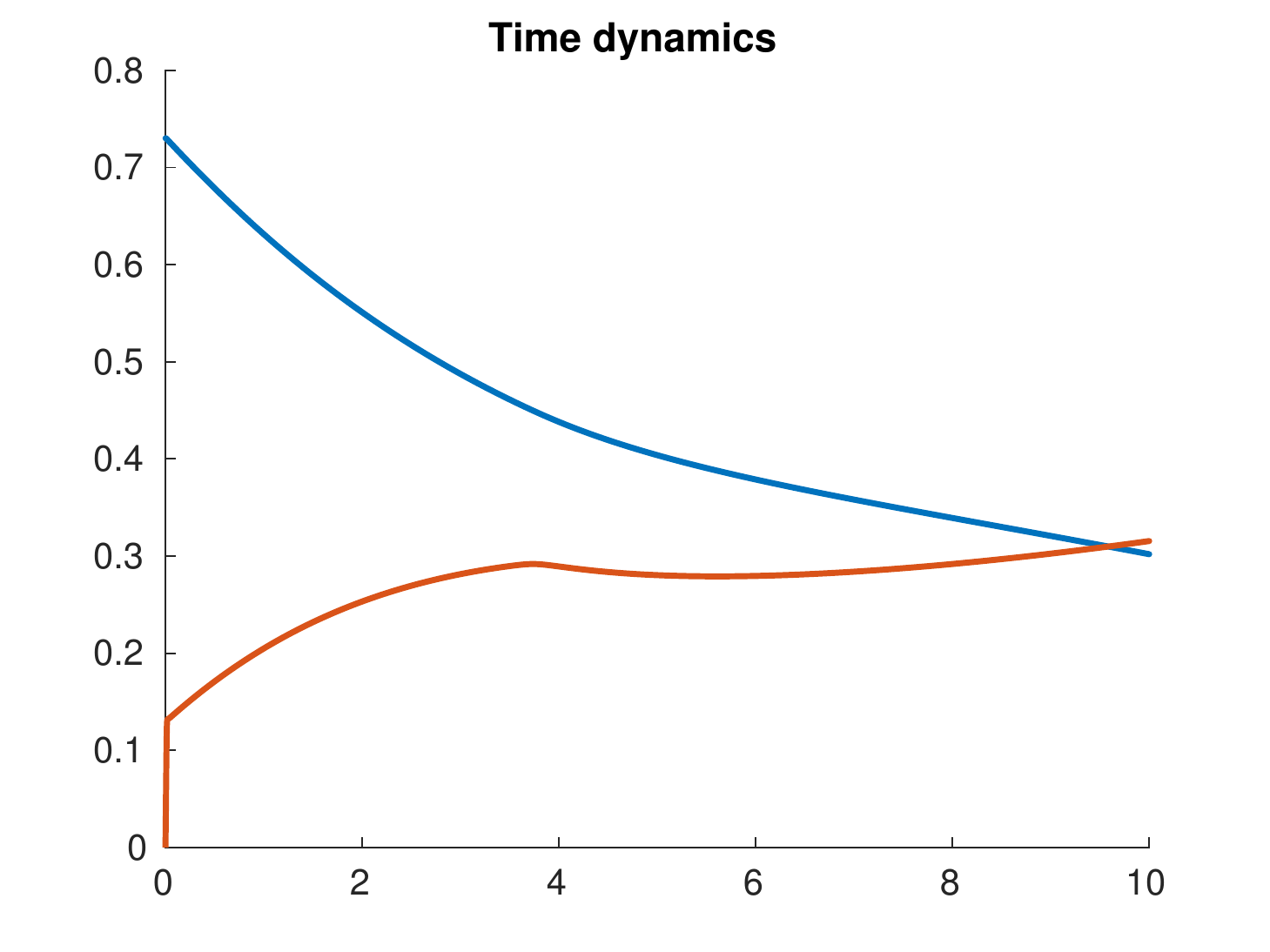}
 \includegraphics[width=.32\textwidth]{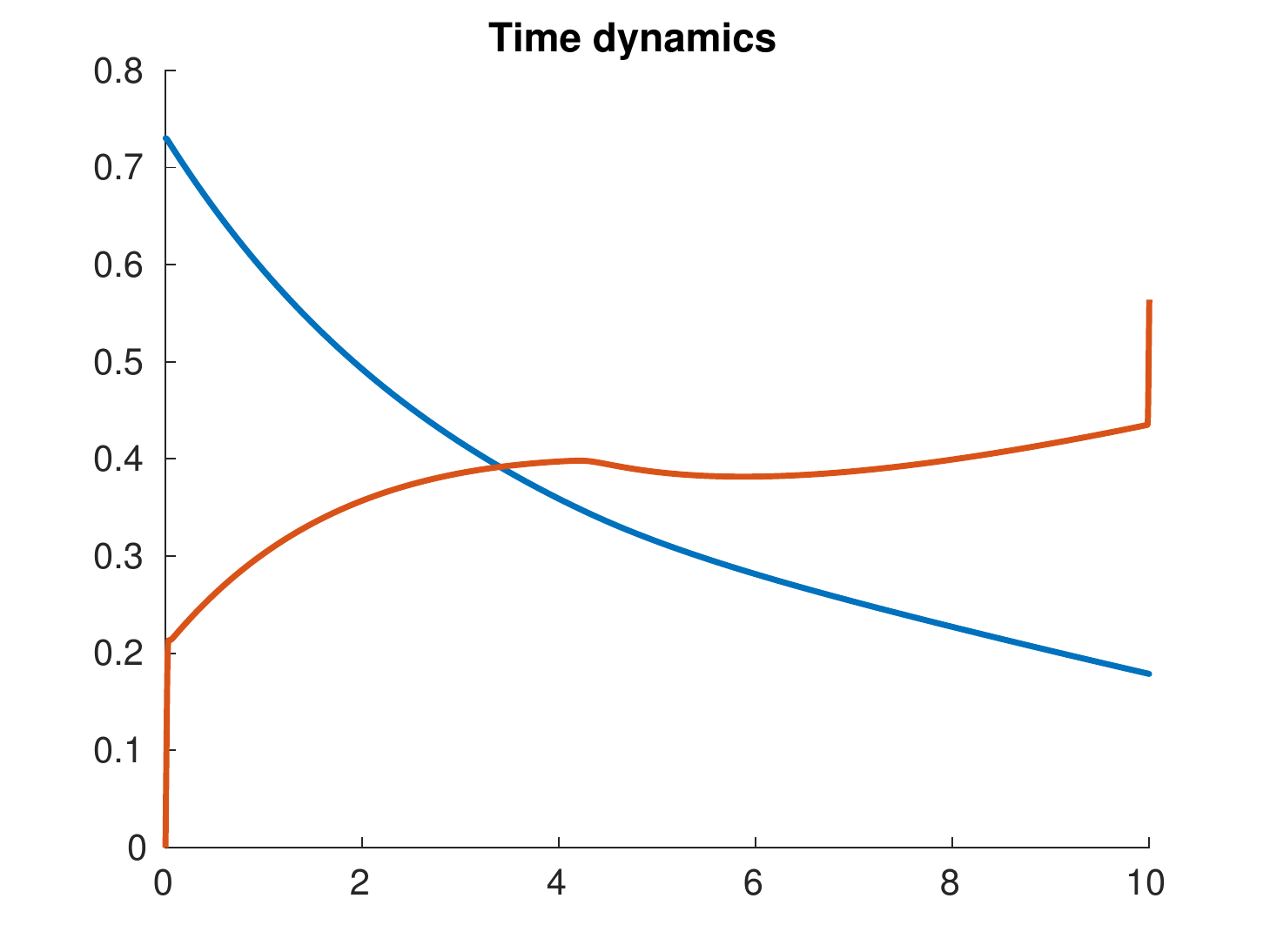}\\
 \includegraphics[width=.32\textwidth]{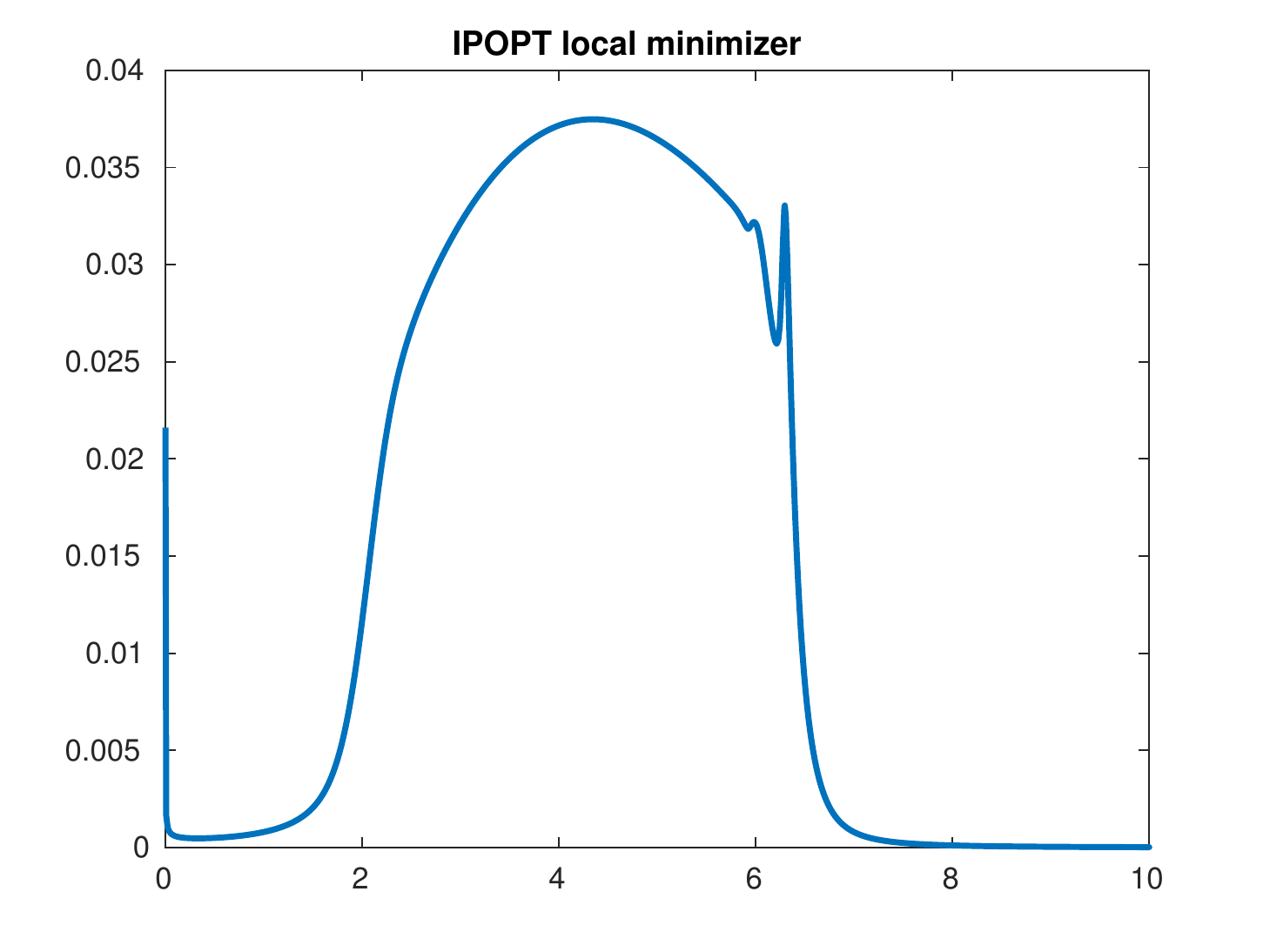} 
 \includegraphics[width=.32\textwidth]{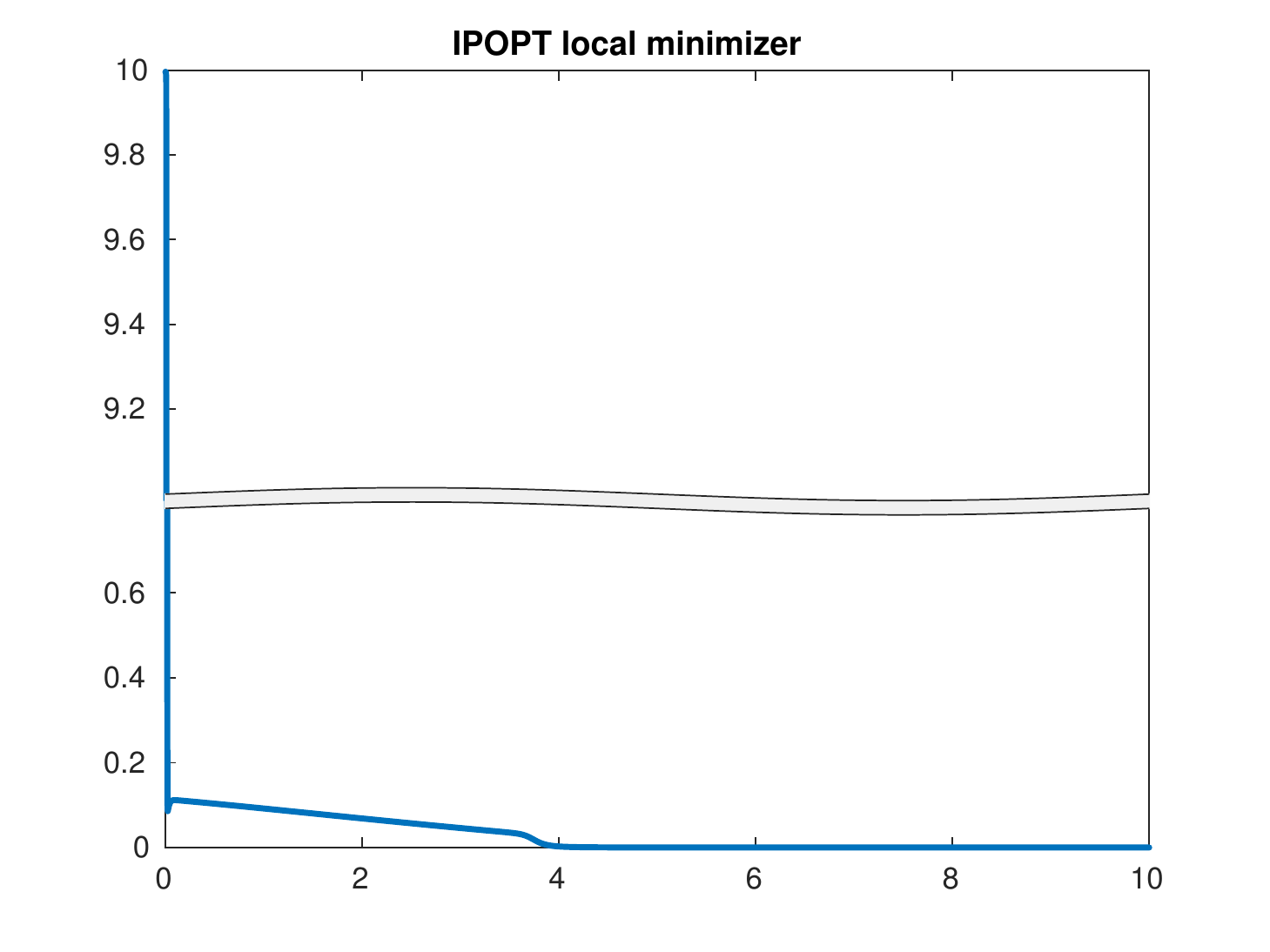} 
 \includegraphics[width=.32\textwidth]{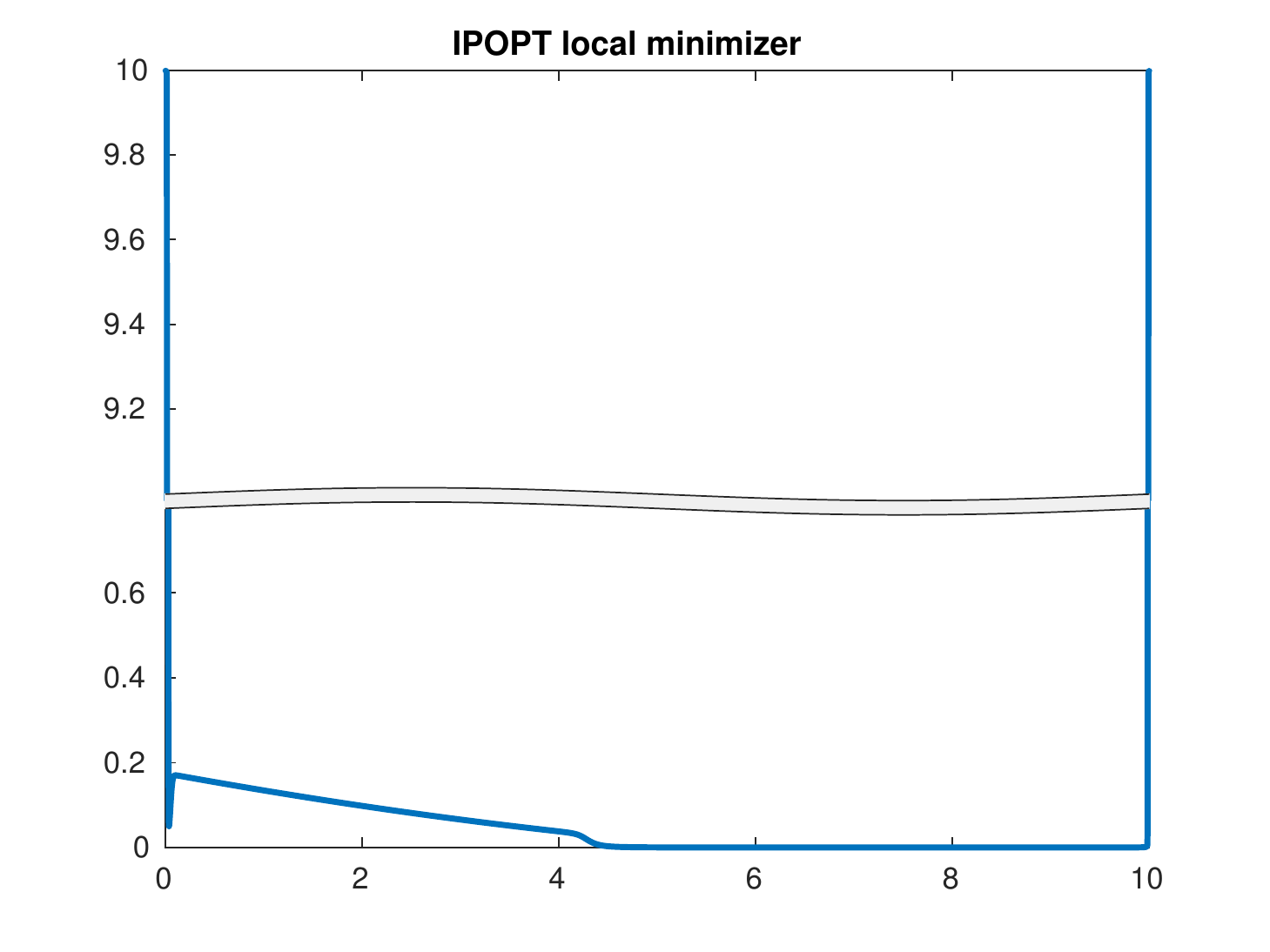}
 \caption{Top: time dynamics (plots of the wild mosquitoes density $n_1$ starting from a positive value {\itshape versus} the {\itshape Wolbachia}-infected mosquitoes density $n_2$ starting from $0$). Bottom: numerical optimal control. From the left to the right: $C = 0.15$, $C = 0.4$ and $C = 0.75$. The parameter $\eps$ is fixed to $1$. \label{fig:numres}}
\end{figure}

\begin{figure}[!ht]
\centering
 \includegraphics[width=.39\textwidth]{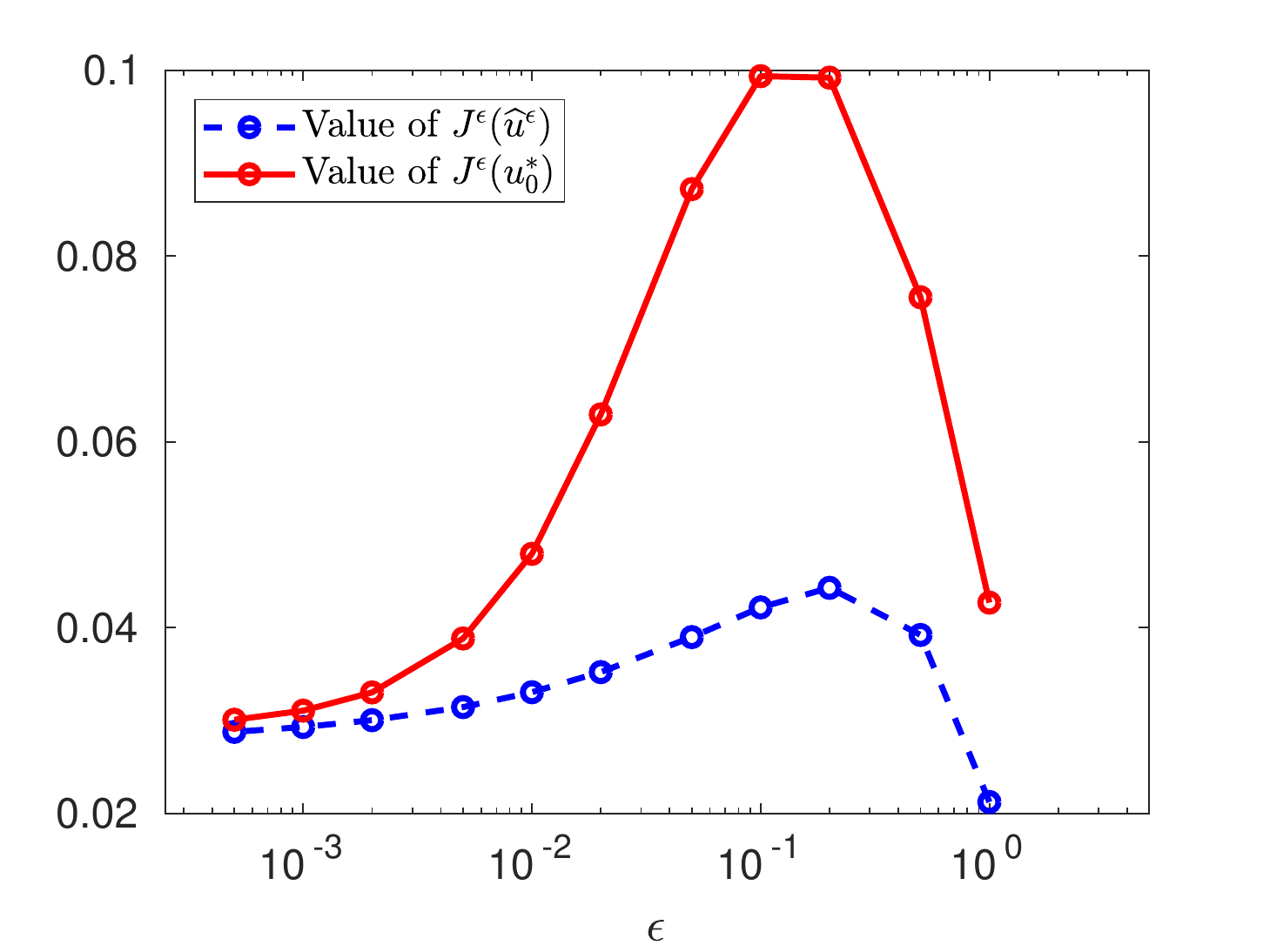}
 \includegraphics[width=.39\textwidth]{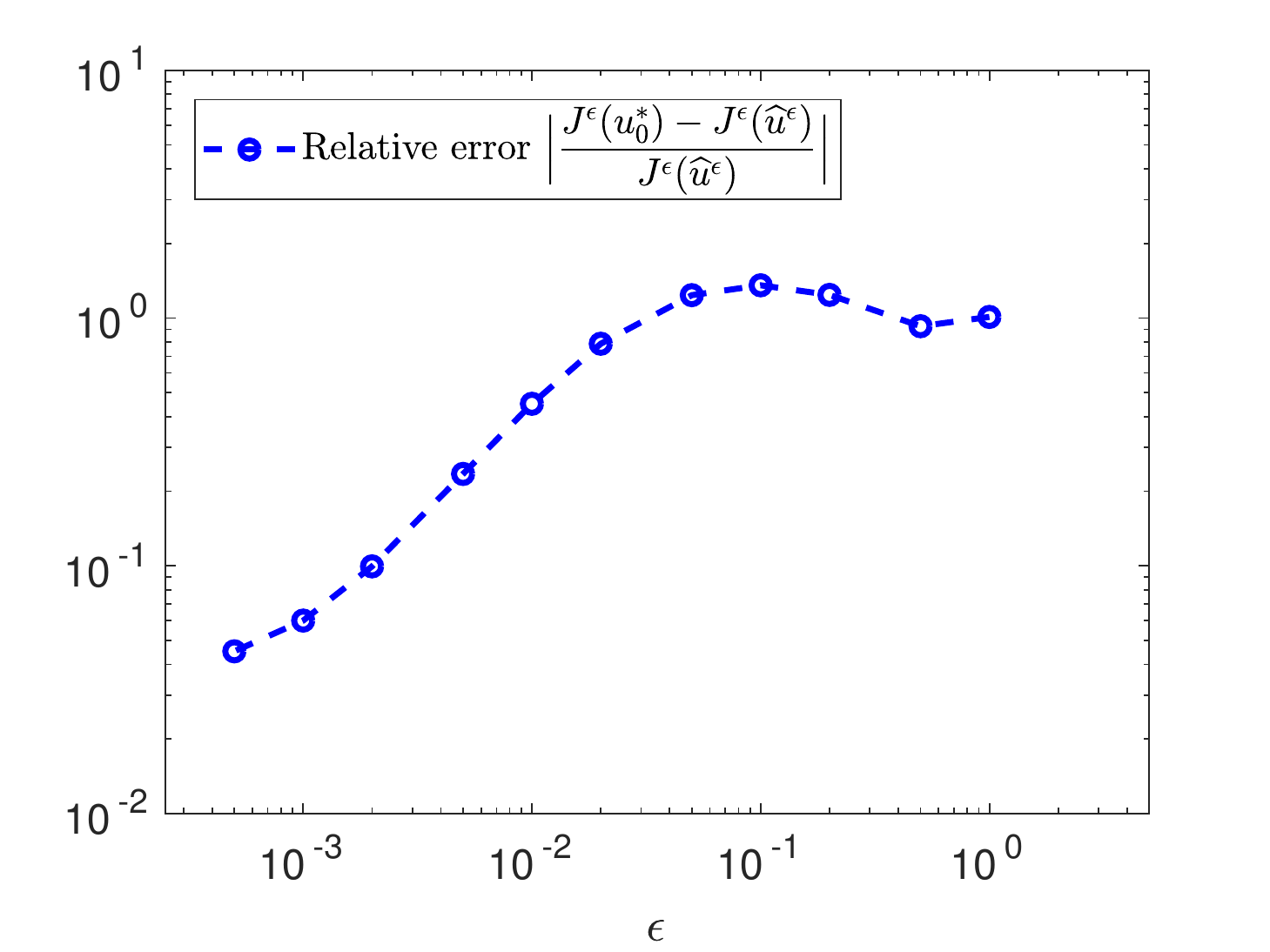}
 \includegraphics[width=.39\textwidth]{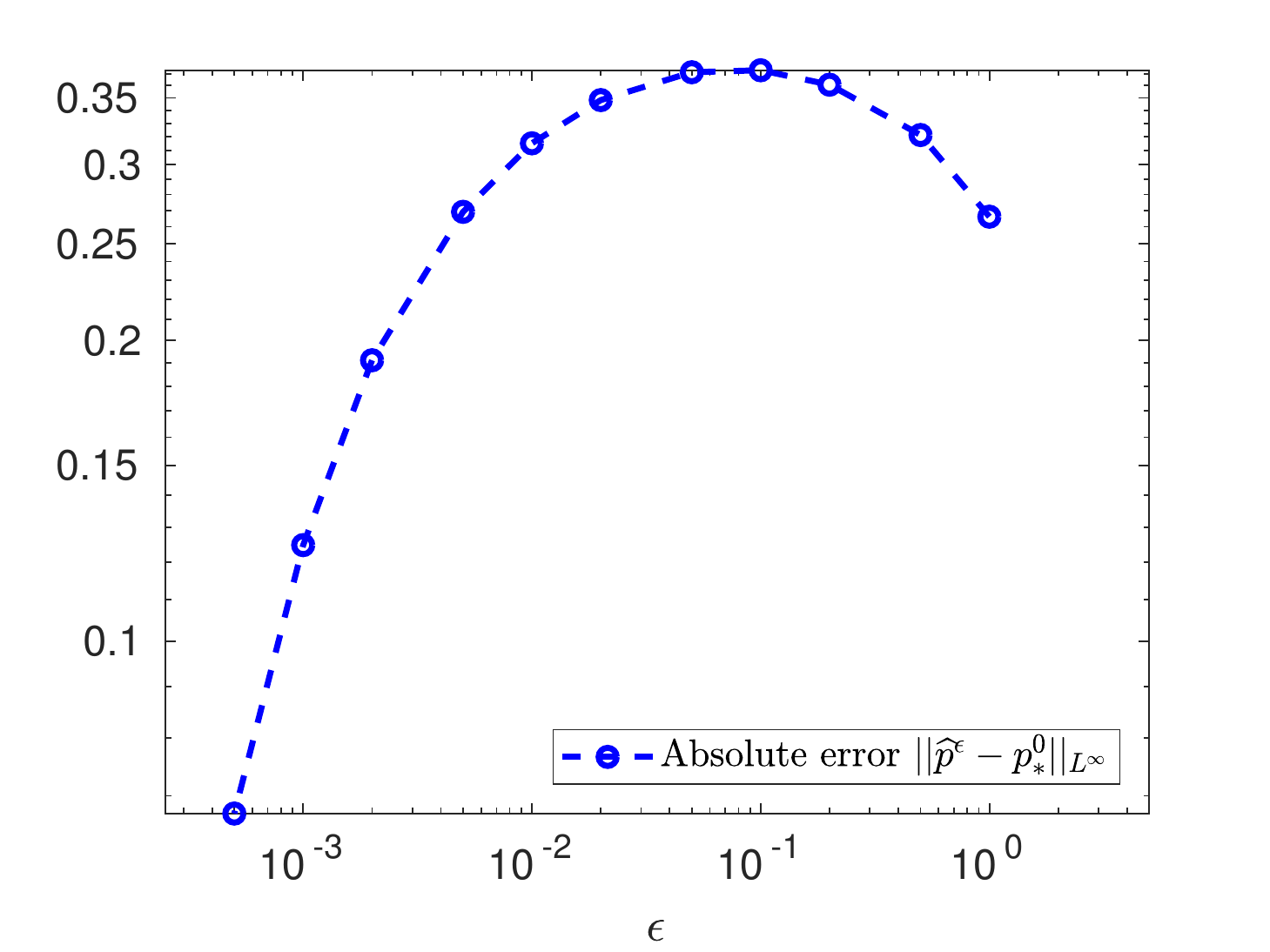}
 \includegraphics[width=.39\textwidth]{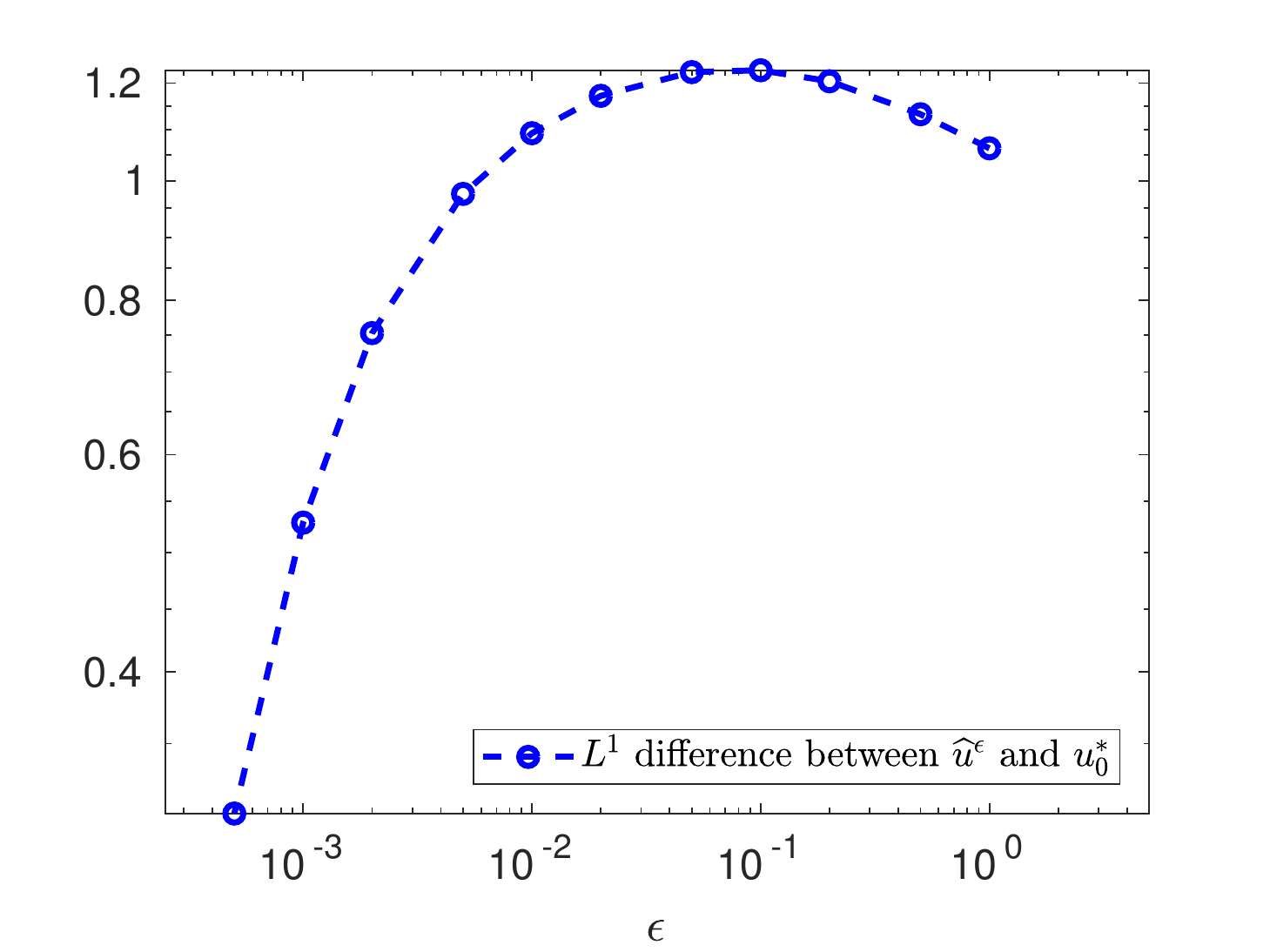}
 \caption{Case $C = 0.75$. Top left: numerical minimum value $J^{\eps}(\widehat{u}^{\eps,\Delta t})$ and $J^{\eps} (u^*_0)$ w.r.t. $\eps$. Top right: relative error between the value of $J^{\eps}$ at the numerical minimizer $\widehat{u}^{\eps,\Delta t}$ and at the exact solution $u^*$ of the asymptotic problem \eqref{prob:reduced} w.r.t. $\eps$. Bottom left: plot of the absolute error between $\widehat{p}^{\eps}$ and $p_0^*$. Bottom right: $L^1$ error between $\widehat{u}^{\eps,\Delta t}$ and $u^*_0$. \label{fig:error75}}
\end{figure}

\begin{figure}[!ht]
\centering
 \includegraphics[width=.39\textwidth]{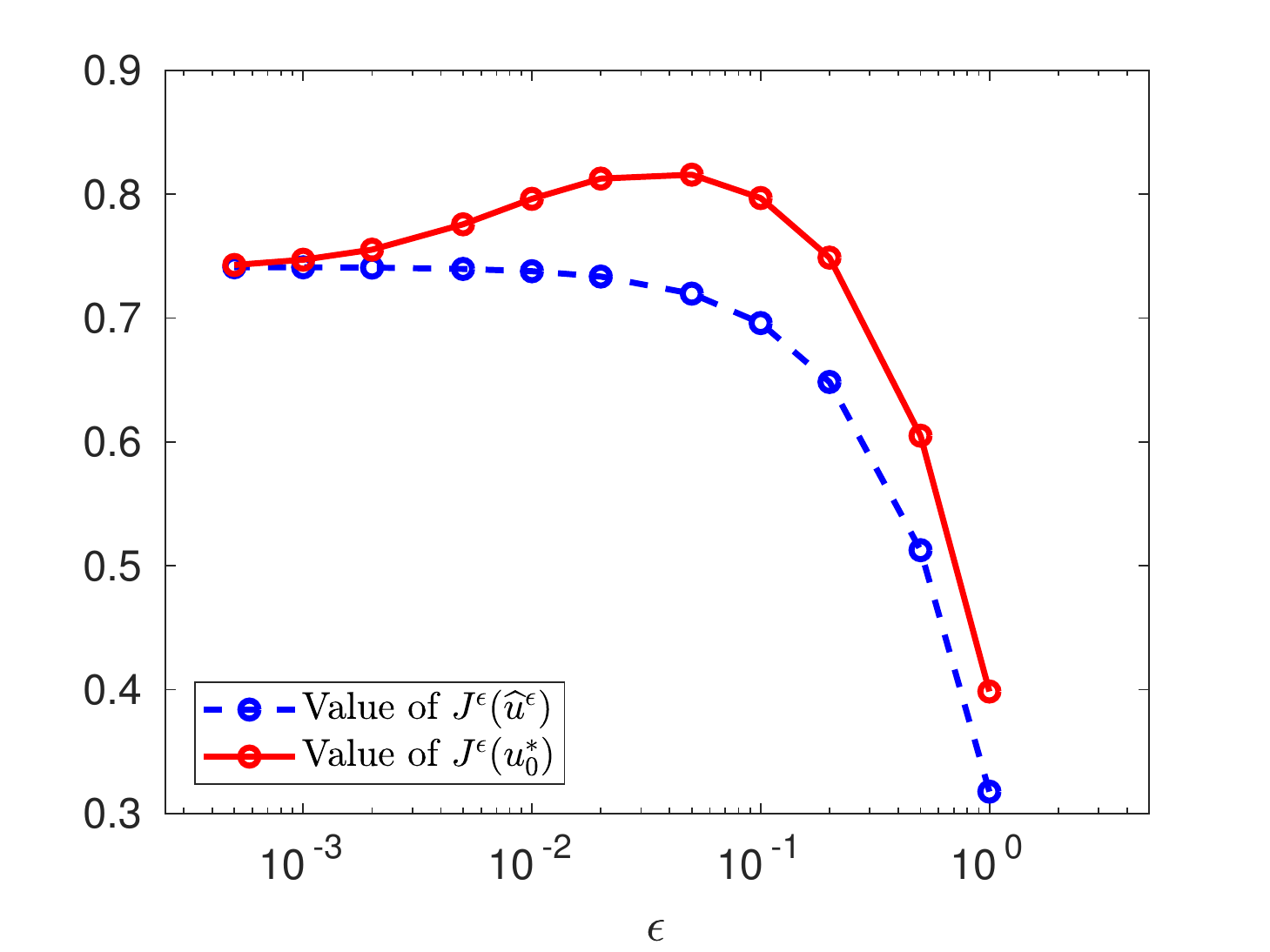}
 \includegraphics[width=.39\textwidth]{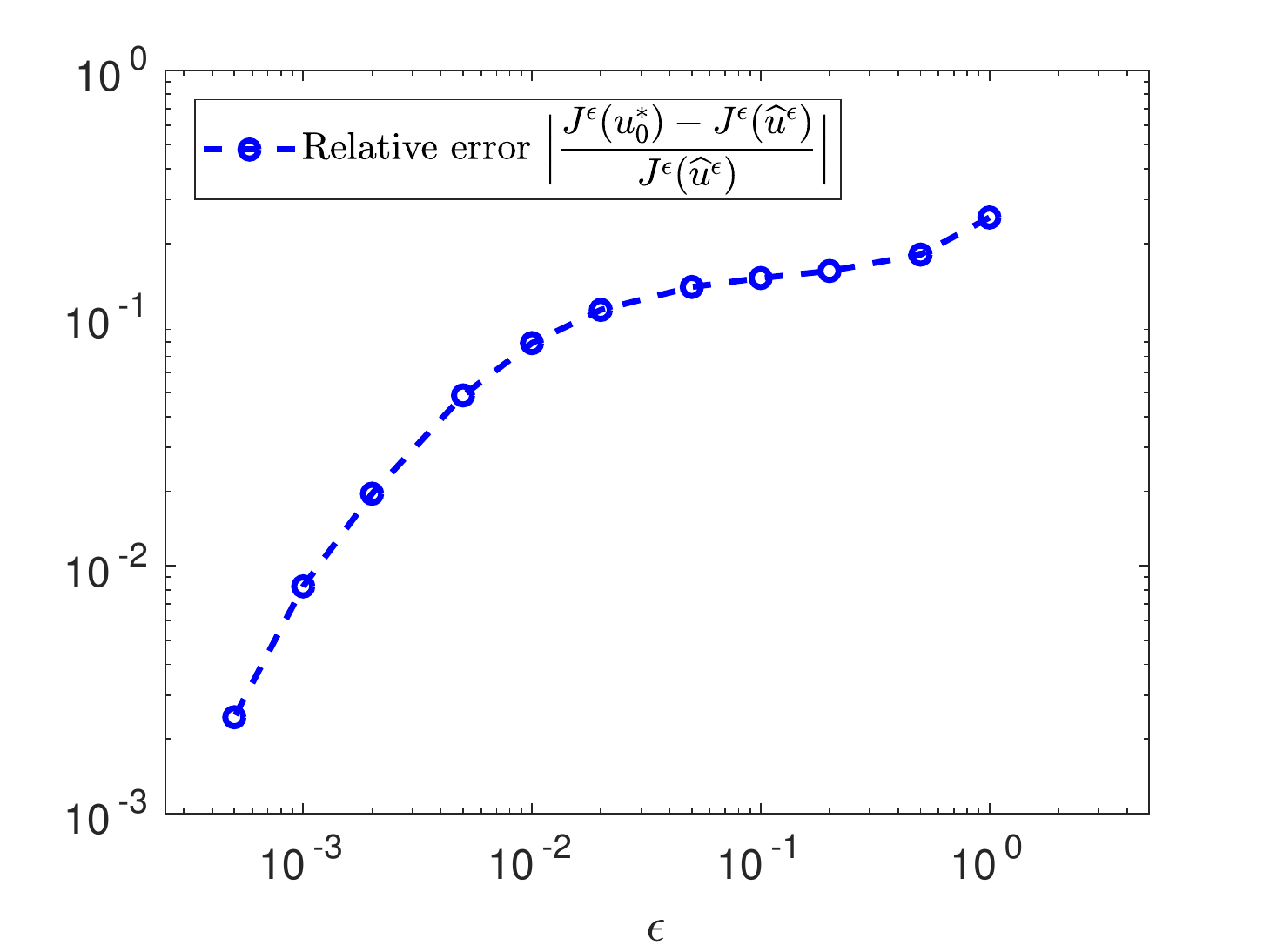}
 \includegraphics[width=.39\textwidth]{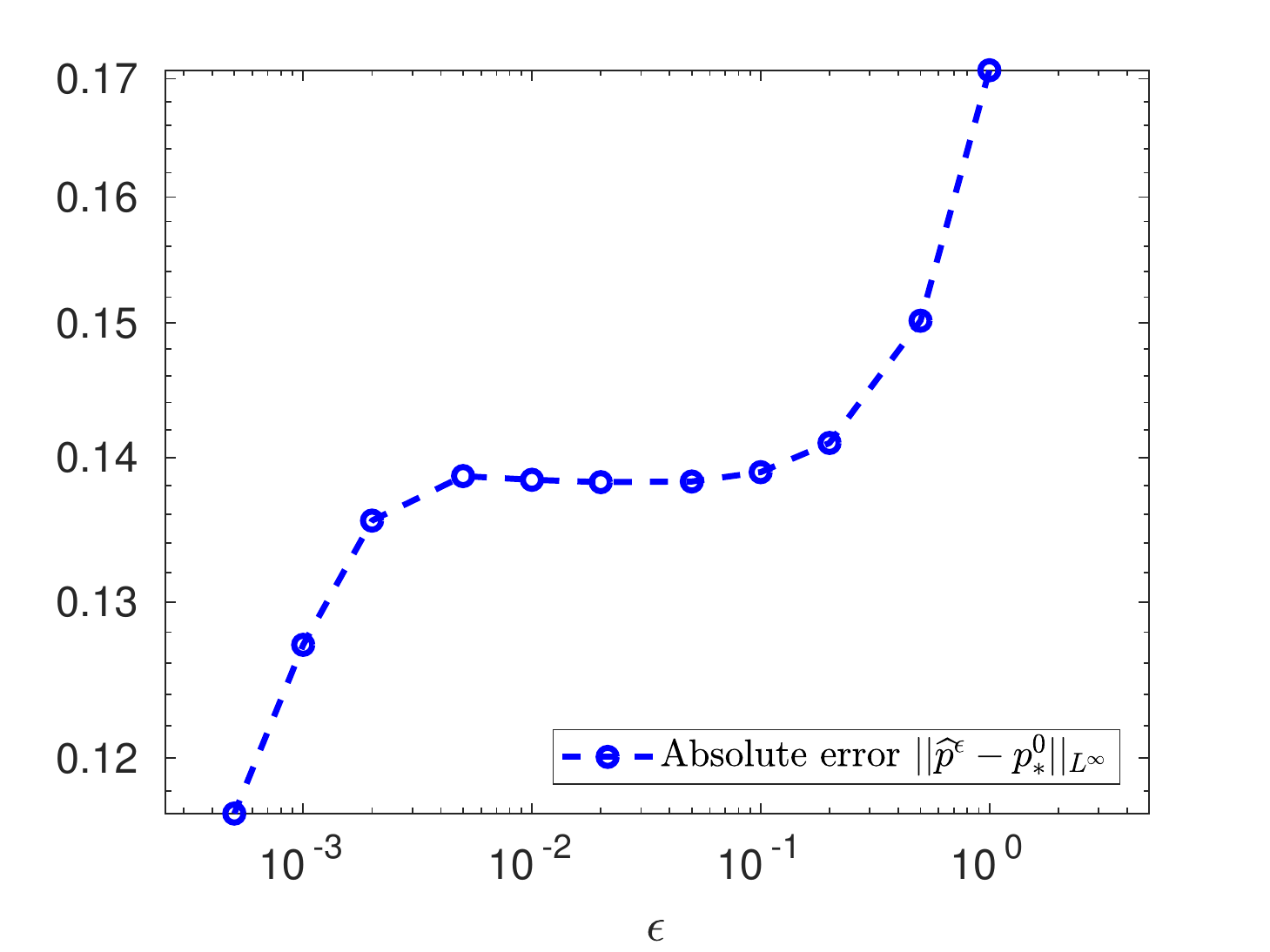}
 \includegraphics[width=.39\textwidth]{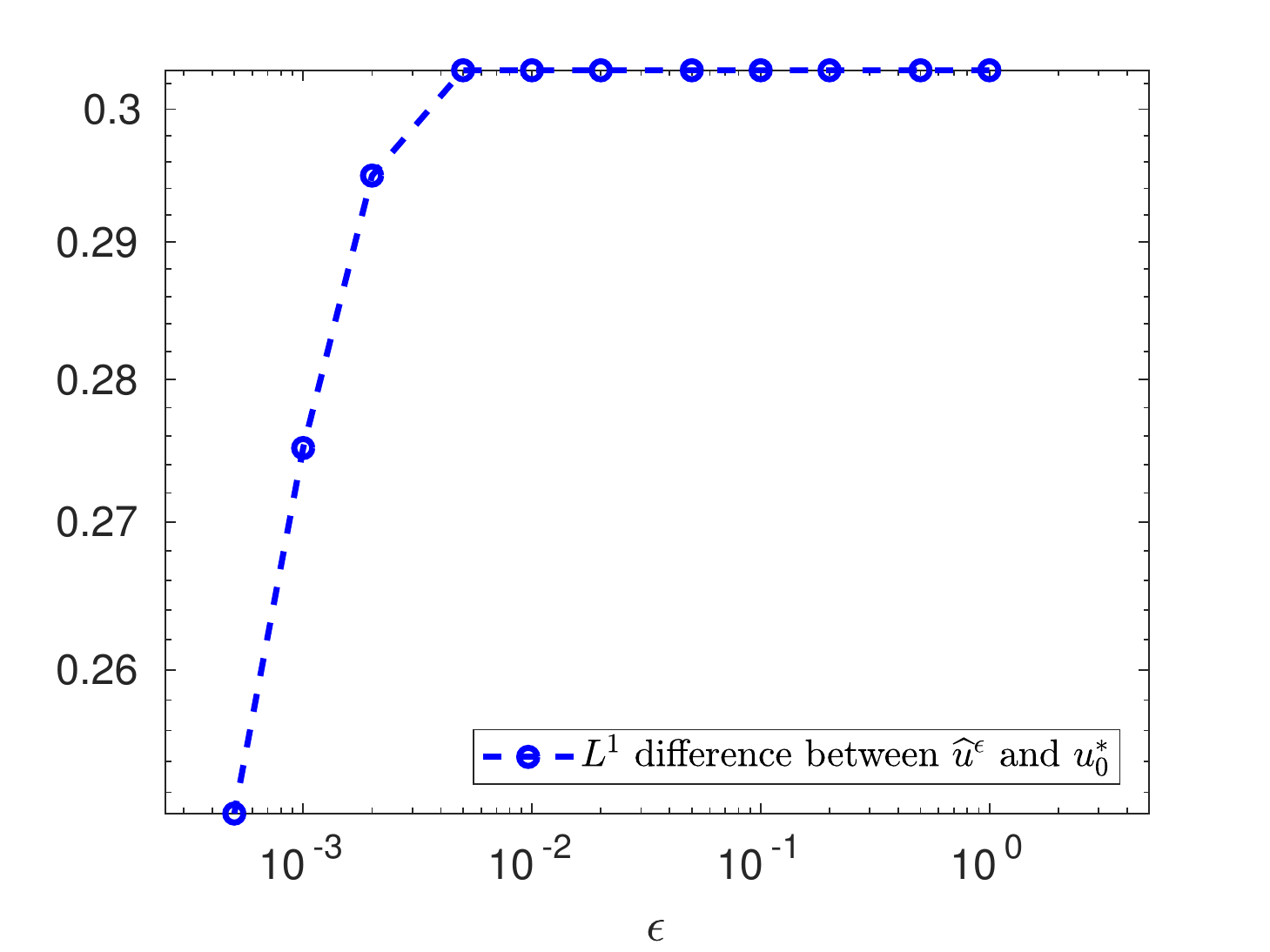}
 \caption{Case $C = 0.15$. Top left: numerical minimum value $J^{\eps}(\widehat{u}^{\eps,\Delta t})$ and $J^{\eps} (u^*_0)$ w.r.t. $\eps$. Top right: relative error between the value of $J^{\eps}$ at the numerical minimizer $\widehat{u}^{\eps,\Delta t}$ and at the exact solution $u^*$ of the asymptotic problem \eqref{prob:reduced} w.r.t. $\eps$. Bottom left: plot of the absolute error between $\widehat{p}^{\eps}$ and $p_0^*$. Bottom right: $L^1$ error between $\widehat{u}^{\eps,\Delta t}$ and $u^*_0$}
 \label{fig:error15}
\end{figure}

Figures \ref{fig:error75} and \ref{fig:error15} are used to validate our approach of considering the asymptotic problem \eqref{prob:reduced} instead of the real one \eqref{prob:full}, with $C = 0.75$ (leading to replacement success) and $C = 0.15$ (leading to replacement failure), respectively. We compare the numerical values of $J(u = \widehat{u}^{\eps,\Delta t})$ obtained either by using the direct optimization routine described above, or by choosing $u = u^*_0$ as the (explicit) solution of Problem \eqref{prob:reduced}. 
As expected, the ratio
\[
 \frac{J^{\eps}(u^*_0) - J^{\eps}(\widehat{u}^{\eps,\Delta t})}{J^{\eps}(\widehat{u}^{\eps,\Delta t})}
\]
visually converges to $0$ as $\eps\searrow 0$. 
The bottom panels in figures \ref{fig:error75} and \ref{fig:error15} illustrate the convergence properties for $p^{\eps}$ stated in Proposition \ref{prop:SFconvergence}, and for $u^{\eps}$ stated in Corollary \ref{cor.strongCV}.

\section{Conclusion}

In this article, we proposed a strategy of {\itshape Wolbachia}-infected mosquitoes releases to control a simplified competitive compartmental system involving wild and infected individuals. Our approach is validated by numerical results that seem promising. Hereafter, we enumerate a list of issues that remain open and will be investigated in a future work.

\paragraph{Partial or complete solving of Problem \eqref{prob:full}.} When investigating numerically this problem (see Section \ref{sec:num}), we observed several interesting properties of minimizers, at least for several relevant values of parameters: {\itshape relaxation} phenomena may appear (meaning that the minimizer $u^*$ is not {\itshape bang-bang} anymore). The set $\{u^*=M\}$ seems to have two connected components meeting 0 and $T$.
 
\paragraph{Asymptotic of Problem \eqref{prob:full} when one makes simultaneously $\eps$ go to zero and $M$ (the pointwise upper-bound constraint on $u$) go to $+\infty$.} According to Theorem \ref{thm:reducedprob}, one shows easily that making successively $\eps$ tend to 0 and then $M$ tend to $+\infty$ yields to a new asymptotic problem whose minimizers are a (typically unique) Dirac mass. When making simultaneously $\eps$ tend to 0 and then $M$ tend to $+\infty$, the behavior of minimizers is not so clear and a careful analysis must be led to understand it.

\paragraph{Investigation of a more realistic model.} Coming back to the initial {\itshape Wolbachia}-infected mosquitoes control problem, it is likely that a model taking into account dispersal effects would provide more satisfying and workable results. To this aim, System \eqref{sys:general} could be replaced by a more general reaction-diffusion system of partial differential equations. It is likely that numerical difficulties may arise for the related optimization problem, needing to develop an adapted approach.

\section*{Acknowledgement.}
The authors were partially supported by the Project ``Analysis and simulation of optimal shapes - application to life-sciences'' of the Paris City Hall.

\appendix

\section{Proof of Lemma \ref{lem.systMonot}}\label{sec.proofs}

Solving the equation $\ff(n_1,n_2)=0$ yields the steady states by direct computation.
Let us use the notations $N = (n_1 + n_2)/K$ and $p = n_2/(n_1 + n_2)$.
The Jacobian associated to the right-hand side $\ff$ of the system reads 
 \begin{multline*}
  \Jac (\nn) = 
  \\
  \begin{pmatrix}
   b_1 \big( (1 -s_h p) (1 - (2-p)N) + s_h p (1 -p) (1 -N) \big) - d_1 & - b_1 (1-p)\big( s_h (1-p) + N (1-s_h) \big) \\
   - b_2 p N & b_2 (1 - (1+p)N) - d_2
 \end{pmatrix}.
 \end{multline*}
It is readily seen that the extra-diagonal terms are non-positive (and even negative if $p \in (0, 1)$ and $N > 0$).
By Kamke-Muller conditions (see \cite{HirSmi.MDS}), this implies that the system is monotone with respect to the cone $\R_+ \times \R_-$, in other words it is competitive.

In particular,
\begin{align*}
 \Jac(n_1^*,0) &= 
 \begin{pmatrix}
  -(b_1 - d_1) & - b_1 + (1 - s_h) d_1 \\
  0 & b_2 d_1/b_1 - d_2
 \end{pmatrix},
 \\
 \Jac(0, n_2^*) &=
 \begin{pmatrix}
  b_1 d_2(1 - s_h) /b_2 - d_1 & 0 \\
  -(b_2 - d_2) & -(b_2 - d_2)
 \end{pmatrix},
\end{align*}
so that conditions \eqref{cond:1} and \eqref{cond:2} easily yield the linear stability of $(n_1^*, 0)$ and $(0, n_2^*)$. Combined with the monotonicity property of the system, we get the asymptotic stability.

Then, $\nn^C$ belongs to the interior of the interval $[(n_1^*, 0), (0, n_2^*)]$ (for the order induced by the comparison principle recalled in Footnote \ref{footnote.comp.pple}), whose bounds are stable steady states, and there is no other steady state in the interior of this interval. Hence it must be unstable, since the dynamics of \eqref{sys:general} is order-preserving.

At $(0, 0)$, we compute the directional derivative in direction $(h,k)$ as
\[
 D\ff (h, k) = \lim_{t \to 0} \frac{\ff(th, tk)}{t} = 
 \begin{pmatrix}
 (b_1 (1 - s_h \frac{k}{h+k}) - d_1) h  \\
 (b_2 - d_2) k
 \end{pmatrix},
\]
and in particular we find that the direction $(0, 1)$ is unstable.



\bibliographystyle{abbrv}
\bibliography{mybibfile}

\end{document}